\documentclass[10pt,reqno, oneside,a4paper]{amsart}
\usepackage{amssymb, amsmath, latexsym}
\usepackage{amsthm}
\usepackage[austrian,english]{babel}
\usepackage{graphicx}
\usepackage{soul}
\usepackage{color}
\usepackage{tikz}
\usetikzlibrary{arrows,automata,backgrounds,calendar,mindmap,trees}
\usepackage{xcolor}

\usepackage{lineno,hyperref}

\hypersetup{colorlinks=true, linkcolor=blue, citecolor=blue}

\numberwithin{equation}{section}
\newtheorem{theorem}{Theorem}[section]
\newtheorem{lemma}[theorem]{Lemma}

\newtheorem{definition}{Definition}[section]
\newtheorem{remark}{Remark}[section]

\usepackage{tikz}
\usetikzlibrary{arrows,automata,backgrounds,calendar,mindmap,trees}
\usepackage{xcolor}

\usepackage[margin=1in]{geometry}

\newcommand{\iO}{\int_{\Omega}}

\title[Entropy method for reaction-diffusion systems]{The entropy method for reaction-diffusion systems \\without detailed balance: \\first order chemical reaction networks}

\begin{document}

\author[K. Fellner, W. Prager, B.Q. Tang]{Klemens Fellner, Wolfgang Prager, Bao Q. Tang}

\address{Klemens Fellner \hfill\break
Institute of Mathematics and Scientific Computing, University of Graz, Heinrichstrasse 36, 8010 Graz, Austria}
\email{klemens.fellner@uni-graz.at}

\address{Wolfgang Prager \hfill\break
Institute of Mathematics and Scientific Computing, University of Graz, Heinrichstrasse 36, 8010 Graz, Austria}
\email{wolfgang.prager@uni-graz.at}

\address{Bao Quoc Tang \hfill\break
Institute of Mathematics and Scientific Computing, University of Graz, Heinrichstrasse 36, 8010 Graz, Austria}
\email{quoc.tang@uni-graz.at}


\begin{abstract}
	In this paper, the applicability of the entropy method for the trend towards equilibrium for reaction-diffusion systems arising from first order chemical reaction networks is studied. 
	In particular, we present a suitable entropy  structure for weakly reversible reaction networks without detail balance condition.
	
	We show by deriving an entropy-entropy dissipation estimate that for any weakly reversible network each solution trajectory converges exponentially fast to the unique positive equilibrium with computable rates. This convergence is shown to be true even in cases when the diffusion coefficients of all but one species are zero. 
	
	For non-weakly reversible networks consisting of source, transmission and target components, it is shown that species belonging to a source or transmission component decay to zero exponentially fast while species belonging to a target component converge to the corresponding positive equilibria, which are determined by the dynamics of the target component and the mass injected from other components. The results of this work, in some sense, complete the picture of trend to equilibrium for first order chemical reaction networks.
\end{abstract}

\maketitle
\numberwithin{equation}{section}
\newtheorem{example}{Example}[section]

\section{Introduction and Main results}
This paper investigates the applicability of the entropy method and proves the convergence to equilibrium for reaction-diffusion systems, which do not satisfy a detailed balance condition.  

The mathematical theory of (spatially homogeneous) chemical reaction networks goes back to the pioneer works of e.g. Horn, Jackson, Feinberg and the Volperts, see \cite{Fe79, Fe87, FeHo74, Ho72, Ho74, HoJa72,Vol72,Vol94} and the references therein. 
The aim is to study the dynamical system behaviour of reaction networks {\it independently of the values of the reaction rates}. It is conjectured since the early of 1970s that in a complex balanced system, the trajectories of the corresponding dynamical system always converge to a positive equilibrium. This conjecture was given the name Global Attractor Conjecture by Craciun et al. \cite{Cra09}. The conjecture in its full generality is -- up to our knowledge -- still unsolved so far, despite many attempts have been made by mathematicians to attack this problem. 

From the many previous works concerning the large time behaviour of chemical reaction networks, the majority of the existing results considers the spatially homogeneous ODE setting. The PDE setting in terms of reaction-diffusion systems is less studied. Also detailed quantitative statements like, e.g. rates of convergence to equilibrium, constitute frequently open questions even in the ODE setting.
\smallskip

Our general aim is to prove quantitative results on the large-time behaviour of chemical reaction networks modelled by reaction-diffusion systems. 
In the present work, we study reaction-diffusion systems arising from first order chemical reaction networks and show that all solution trajectories converge exponentially to corresponding equilibria with explicitly computable rates.
\smallskip

Our approach applies the so called entropy method.
Going back to ideas of Boltzmann and Grad, the fundamental idea of the entropy method is to quantify the monotone decay of a suitable entropy  (e.g. a convex Lyapunov) functional 
in terms of a \emph{functional inequality} connecting the time-derivative of the entropy, the so called entropy dissipation functional, 
back to the entropy functional itself, i.e. to derive a so called \emph{entropy entropy-dissipation (EED)} inequality. Such an EED inequality can only hold provided that all conservation laws are taken into account. After having established an EED inequality and applying it to global solutions of a dissipative evolutionary problem, a direct Gronwall argument implies convergence to equilibrium in relative entropy with rates and constants, 
which can be made explicit. 

By being based on functional inequalities (rather than on direct estimates on the solutions), a major advantage of the entropy method is its robustness with respect to model variations and generalisations. 
Moreover,  the entropy method is per se a fully nonlinear approach. 
\smallskip

The fundamental idea of the entropy method originates from the pioneer works of kinetic theory and from names like Boltzmann and Grad in order to 
investigate the trend to equilibrium of e.g. models of gas kinetics. 

A systematic effort in developing the entropy method for dissipative evolution equations started not until much later, 
see e.g. the seminal works \cite{tos96,toscani_villani1,CJMTU,AMTU,DVinhom1} and the references therein 
for scalar (nonlinear) diffusion or Fokker-Planck equations, and in particular 
the paper of Desvillettes and Villani concerning the trend to equilibrium for the spatial inhomogeneous Boltzmann equation \cite{DVinhom2}.  
The derivation of  EED inequalities for scalar evolution equations is typically based on the Barky-Emery strategy (see e.g. \cite{CJMTU,AMTU}), 
which seems to fail (or be too involved) to apply to systems.
\smallskip

The great challenge of the entropy method for systems is, therefore, to be able to derive an entropy entropy-dissipation inequality, which summarises (in the sense of measuring with a convex entropy functional) the entire dissipative behaviour of solutions to a (possibly nonlinear) dynamical system to which the EED inequality shall be applied to. Preliminary results based on a (non-explicit) compactness-contradiction argument in 2D were obtained e.g. in \cite{Gro92,GGH96,GH97} in the context semiconductor drift-diffusion models. 

The first proof of an EED inequality with explicitly computable constants and rates
for specific nonlinear reaction-diffusion systems was shown in \cite{DeFe06}
and followed by e.g. \cite{DeFe_Con, DeFe08, DeFe15, BaFeEv14, MiHaMa14}. 
The application of these EED inequalities to global solutions of the corresponding reaction-diffusion systems proves (together with Csisz\'ar-Kullback-Pinsker type inequalities) the explicit convergence to equilibrium for these reaction-diffusion systems. 

We emphasise that all these previous results on entropy methods for systems assumed a {\it detailed balance condition} and, thus, features the free energy functional as a natural convex Lyapunov functional. 
\medskip

A main novelty of the paper lies in demonstrating how the entropy method can be generalised to first order reaction networks without detailed balance equilibria. In particular we shall consider firstly \emph{weakly reversible networks} and secondly even more general \emph{composite systems consisting of source, transmission and target components} (see below for the 
precise definitions). 

We feel that it is important to point out that while there are certainly many classical approaches by which linear reaction-diffusion systems can be successfully dealt with, our task at hand is to clarify the entropic structure and the applicability of the entropy method for linear reaction networks as a first step before being able to turn to nonlinear problems in the future. 
{See \cite{DFT} for such a generalisation of the method to nonlinear reaction-diffusion systems satisfying the so-called complex balance condition (see Definition \ref{ComplexBalance} below).} 
\medskip

The goal of this present work is to prove the explicit convergence to equilibrium for the complex balanced and more general reaction-diffusion systems corresponding to first order reaction networks. To be more precise, we study first order reaction networks of the form
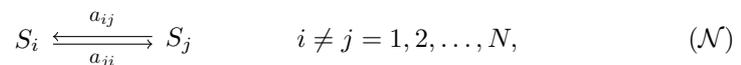
\begin{figure}[htp]
\begin{flushright}
\scalebox{1}[1]{
\begin{tikzpicture}
\node (a) {$S_{i}$} node (b) at (2,0) {$S_{j}$} node (c) at (5,0) {$i\not=j = 1,2,\ldots,N,$} node (d) at (9,0) {$(\mathcal{N})$};
\draw[arrows=->] ([xshift =0.5mm, yshift=-0.5mm]a.east) -- node [below] {\scalebox{.8}[.8]{$a_{ji}$}} ([xshift=-0.5mm,yshift=-0.5mm]b.west);
\draw[arrows=->] ([xshift =-0.5mm,yshift=0.5mm]b.west) -- node [above] {\scalebox{.8}[.8]{$a_{ij}$}} ([xshift=0.5mm,yshift=0.5mm]a.east);
\end{tikzpicture}} \caption{A first-order chemical reaction network}\label{Reaction}
\end{flushright}
\end{figure}\\
where $S_i, i = 1,2,\ldots, N$, are different chemical substances (or species) and $a_{ij}, a_{ji} \geq 0$ are reaction rate constants. In particular, $a_{ij}$ denotes the reaction rates from the species $S_j$ to $S_i$. 

First order reaction networks appear in many classical models, see e.g. \cite{Smo,Rot}. More recently, first order catalytic reactions are used to model transcription and translation of genes in \cite{TVO}. The evolution of the surface morphology during epitaxial growth involves the nucleation and growth of atomic islands, and these processes may be described by first order adsorption and desorption reactions coupled with diffusion along the surface. A first order reaction network can also be used to describe the reversible transitions between various conformational states of proteins (see e.g. \cite{Mayor03}). RNA also exists in several conformations, and the transitions between various folding states follow first order kinetics (see \cite{Bok03}). 
\medskip

In the present paper, we investigate the entropy method and the trend to equilibrium of reaction-diffusion systems modelling first order reaction networks with mass action kinetics.
More precisely, we shall consider the reaction network $\mathcal{N}$ in the context of reaction-diffusion equations and 
assume that  for all $i=1,2,\ldots,N$ the substances $S_i$ are described by spatial-temporal concentrations $u_i(x,t)$ at position $x\in \Omega$ and time $t\geq 0$. 
Here, $\Omega$ shall denote a bounded domain 
$\Omega \subset \mathbb{R}^n$ with sufficiently smooth boundary $\partial\Omega$ (that is 
$\partial\Omega\in C^{2+\alpha}$ to avoid all difficulties with boundary regularity, although the below methods should equally work under weaker assumptions) and the outer unit normal $\nu(x)$ for all $x\in\partial\Omega$.
Due to the rescaling $x\to |\Omega|^{1/n} x$, we can moreover 
consider (without loss of generality) domains with normalised volume, i.e.
$$
|\Omega|=1.
$$

In addition, we assume that each substance $S_i$ diffuses with a diffusion rate $d_i\geq 0$ for all $i=1,2,\ldots,N$. 
Finally, we shall assume mass action law kinetics as model for the reaction rates, which leads to the following linear reaction-diffusion system:
\begin{equation}\label{VectorSystem}
	\begin{cases}
		X_t = D\Delta X + AX, &\qquad x\in\Omega, \qquad t>0,\\
		\partial_{\nu}X = 0, &\qquad x\in\partial\Omega, \qquad t>0,\\
		X(x,0) = X_0(x)\ge0, &\qquad x\in\Omega,
	\end{cases}
\end{equation}
where $X(x,t) = [u_1(x,t), u_2(x,t), \ldots, u_N(x,t)]^{T}$ denotes the vector of concentrations subject to non-negative initial conditions $X_0(x) = [u_{1,0}(x)\ge0, u_{2,0}(x)\ge0, \ldots, u_{N,0}(x)\ge0]^T$, $D = \text{diag}(d_1,d_2,\ldots,d_N)$ denotes the diagonal diffusion matrix and the reaction matrix $A = (a_{ij}) \in \mathbb{R}^{N\times N}$ satisfies the following conditions:
\begin{equation}\label{a_jj}
\begin{cases}
	a_{ij} \geq 0, &\qquad \text{for all } i\not=j, \quad i,j =1,2,\ldots,N,\\ 
	a_{jj} = -\sum_{i=1,i\not= j}^{N}a_{ij}, &\qquad \text{for all } j =1,2,\ldots,N.
\end{cases}
\end{equation}

The conditions \eqref{a_jj} on the reaction matrix $A$ imply in particular that the vector $(1,1,\ldots,1)^T$ constitutes a left-eigenvector 
corresponding to the eigenvalue zero. Together with homogeneous Neumann boundary conditions this implies that solutions to \eqref{VectorSystem} admit the following {\it conservation of total mass}\,:
\begin{equation}\label{MassConservation}
	\sum_{i=1}^{N}\int_{\Omega}u_i(x,t)dx = \sum_{i=1}^{N}\int_{\Omega}u_{i,0}(x)dx =: M>0, \qquad \text{ for all } t>0,
\end{equation}
where $M>0$ is the {\it initial total mass}, which we shall assume positive. 

If $X(x,t)\equiv X(t)$, then system \eqref{VectorSystem} reduces to the corresponding space-homogeneous ODE model. 
Independently of PDE- or ODE-setting, we recall the
following definitions of equilibria from e.g. \cite{HoJa72,Fe79,Vol94}.

\begin{definition}[Homogeneous Equilibrium]\label{Equilibria}\hfill\\
A 
state $X_{\infty} = (u_{1,\infty}, u_{2,\infty}, \ldots, u_{N,\infty})$ is called a 
homogeneous {equilibrium} or shortly equilibrium of the first order reaction network $\mathcal{N}$ if $AX_{\infty}=0$.
\end{definition}
\begin{definition}[Detailed Balance Equilibrium]\label{DetailedBalance}\hfill\\
A positive equilibrium state $X_{\infty} = (u_{1,\infty}, u_{2,\infty}, \ldots, u_{N,\infty})>0$ is called a {detailed balance equilibrium} for the reaction network $\mathcal{N}$ if a positive reaction rate constant $a_{ij}>0$ for $i\neq j$ implies also a positive reversed reaction rate constant $a_{ji}>0$ and that the forward and backward reaction rates balance at equilibrium, i.e. 
	$$
		a_{ji}u_{i,\infty} = a_{ij}u_{j,\infty}
	$$
The reaction network $\mathcal{N}$ is called to satisfy the detailed balance condition if it admits a detailed balance equilibrium.
\end{definition}

\begin{definition}[Complex Balance Equilibrium]\label{ComplexBalance}\hfill \\
A positive equilibrium state $X_{\infty} = (u_{1,\infty}, u_{2,\infty}, \ldots, u_{N,\infty})>0$ is called a complex balance equilibrium for the reaction network $\mathcal{N}$ if for all $k=1,2,\ldots,N$, the total in-flow into the substance $S_k$ balances in equilibrium the total 
	out-flow from $S_k$ to all other substances $S_i$, i.e.
	$$
		\sum_{\{1\leq i\leq N:\; a_{ki}>0\}}a_{ki}u_{i,\infty} = \biggl(\sum_{\{1\leq j\leq N:\; a_{jk}>0\}}a_{jk}\biggr)u_{k,\infty}.
	$$
The reaction network $\mathcal{N}$ is called complex balanced if it admits a complex balance equilibrium. 
Moreover for complex balanced chemical reaction networks, all equilibria are complex balanced, see e.g. \cite{Ho72}. 

\end{definition}

\begin{example}[Detailed balance equilibria are complex balance equilibria]\hfill \\
It is easy to see that detailed balance equilibria are also complex balance equilibria while the reverse does not hold in general, even for reversible networks. For example, consider the reaction network in Figure \ref{ReversibleNet},
\begin{figure}[htp]
\begin{center}
\scalebox{1}[1]{
\begin{tikzpicture}
 \node (a) {$S_1$} node (b) at (2,0) {$S_2$} node (c) at (0,-2) {$S_3$};
 \draw[arrows=->] ([xshift =0.5mm, yshift=-0.5mm]a.east) -- node [below] {\scalebox{.8}[.8]{$a_{21}$}} ([xshift=-0.5mm,yshift=-0.5mm]b.west);
 \draw[arrows=->] ([xshift =-0.5mm,yshift=0.5mm]b.west) -- node [above] {\scalebox{.8}[.8]{$a_{12}$}} ([xshift=0.5mm,yshift=0.5mm]a.east);
\draw[arrows=->] ([xshift=-.5mm]a.south) -- node [left] {\scalebox{.8}[.8]{$a_{31}$}} ([xshift=-.5mm]c.north);
\draw[arrows=->] ([xshift=.5mm]c.north) -- node [right] {\scalebox{.8}[.8]{$a_{13}$}} ([xshift=.5mm]a.south);
\draw[arrows=->] ([xshift=1mm]c.north east) -- node[above left] {\scalebox{.8}[.8]{$a_{23}$}} ([xshift=.5mm]b.south west);
\draw[arrows=->] ([xshift=-1mm]b.south) -- node[below right] {\scalebox{.8}[.8]{$a_{32}$}} ([xshift=1mm, yshift=1mm]c.east);
\end{tikzpicture}}\caption{A reversible network}\label{ReversibleNet}
\end{center}
\end{figure}
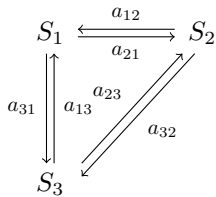
where all reaction rates constants $a_{ij}>0$ are assumed positive and the network is thus fully reversible. 
The corresponding reaction-diffusion system with homogeneous Neumann boundary conditions
\begin{equation}\label{3x3}
	\begin{cases}
		\partial_tu_1 - d_1\Delta u_1 = -(a_{21}+a_{31})u_1 + a_{12}u_2 + a_{13}u_3,\\
		\partial_tu_2 - d_2\Delta u_2 = a_{21}u_1 -(a_{12}+a_{32})u_2 + a_{23}u_3,\\
		\partial_tu_3 - d_3\Delta u_3 = a_{31}u_1 + a_{32}u_2 -(a_{13}+a_{23})u_3,\\
		\partial_{\nu} u_1 = \partial_{\nu} u_2 = \partial_{\nu} u_3 = 0.
 \end{cases}
\end{equation}
exhibits the constant equilibrium $X_{\infty} = (u_{1,\infty}, u_{2,\infty}, u_{3,\infty})$ satisfying $AX_{\infty}=0$, i.e.
\begin{equation}\label{equi3x3} 
	\begin{cases}
		a_{12}u_{2,\infty} + a_{13}u_{3,\infty} = (a_{21}+a_{31})u_{1,\infty},\\
		a_{21}u_{1,\infty}  + a_{23}u_{3,\infty} = (a_{12}+a_{32})u_{2,\infty},\\
		a_{31}u_{1,\infty} + a_{32}u_{2,\infty} = (a_{13}+a_{23})u_{3,\infty},
	\end{cases}
\end{equation}
which has a unique nontrivial solution once the mass conservation \eqref{MassConservation} is taken into account. 

According to Definition \ref{ComplexBalance}, it is clear that system \eqref{equi3x3} constitutes  a complex balance equilibrium for all reaction rate constants $a_{ij}>0$.
For $X_{\infty}$ to be a detailed balance equilibrium, however, it is additionally necessary that
\begin{equation}\label{detailbalance3x3}
	\begin{cases}
	a_{12}u_{2,\infty} = a_{21}u_{1,\infty},\\
	a_{23}u_{3,\infty} = a_{32}u_{2,\infty},\\
	a_{31}u_{1,\infty} = a_{13}u_{3,\infty},
	\end{cases}
\end{equation}
which obviously implies \eqref{equi3x3}. Yet the equations \eqref{detailbalance3x3} can only have a solution if 
\begin{equation}\label{detailbalance3x3cond}
	\frac{a_{12}\cdot a_{23}\cdot a_{31}}{a_{21}\cdot a_{32}\cdot a_{13}} = 1,
\end{equation}
holds; in other words if the product of the reaction rate constants multiplied in the clockwise sense of the above reaction network graph equals the product of the reaction rate constants multiplied in the counterclockwise sense. The condition \eqref{detailbalance3x3cond} 
is thus necessary and sufficient for system \eqref{3x3} to admit a detailed balance equilibrium. 
\end{example}

\begin{remark}[General definition of detailed and complex balance]\hfill\\
The concepts of detailed balance and complex balance are also defined for general higher order chemical reaction networks, see e.g. \cite{Ho72}. For simplicity, we stated here the definition corresponding to the first order network $\mathcal{N}$. In general, one can roughly say that a state $X_{\infty}$ is called a complex balanced equilibrium if at equilibrium the total in-flow to each specie $S_i$ is equal to the total out-flow from $S_i$.
\end{remark}
\begin{remark}[Detailed balance and reversibility]\hfill\\
It follows from Definition \eqref{DetailedBalance} that if $\mathcal{N}$ satisfies the detailed balance condition, then it is also reversible in the sense that for any reaction $S_i \rightarrow S_j$ also the reverse reaction $S_j \rightarrow S_i$ takes place.
\end{remark}
\medskip

The set of complex balanced systems is much larger than the one of detailed balance systems. Horn already gave necessary and sufficient conditions for a network to satisfy the complex balance condition in \cite{Ho72}.
For convenience of the reader, we present in the following the associated definitions of directed graphs as representations of reaction networks. The image of the associated graphs will also help following some of our main estimates. 
\smallskip

A directed graph $G$ corresponding to a given  reaction network $\mathcal{N}$ is defined by considering the 
substances $S_i, i = 1,2,\ldots, N,$ as the $N$ nodes of $G$, which are connected for all $i\not=j = 1,2,\ldots, N$ 
by an edge with starting node $S_i$ and finishing node $S_j$ if and only if  the reaction $S_i \xrightarrow{a_{ji}} S_j$ occurs with a positive reaction rate constant $a_{ji}>0$.

\begin{definition}[Linkage classes partition of a first order reaction network, Connected networks]\label{linkageclass}\hfill\\
\textcolor{black}{
A linkage class  $\mathcal{L}$ of a first order network $\mathcal{N}$ is a maximal set of connected substances, i.e. 
$S_i, S_j \in \mathcal{L}$ implies that $S_i$ and $S_j$ are connected (in the sense {that there exist $S_i \equiv S_{r_1},S_{r_2}\ldots, S_{r_{k-1}}, S_{r_k} \equiv S_j$ such that for each $1\le \ell\le k-1$, either the reaction $S_{r_\ell} \to S_{r_{\ell+1}}$ or $S_{r_{\ell+1}}\to S_{r_{\ell}}$ happens}) but  $S_i \in \mathcal{L}$ and 
$S_j \not\in \mathcal{L}$ implies that $S_i$ and $S_j$ are not connected. 
}

\textcolor{black}
{If a reaction network consists only of one linkage class, we shall call such a network \emph{connected.}} 
\end{definition}

\begin{definition}[Weak reversibility of a first order reaction network]\label{weaklyreversible}\hfill\\
A {\color{black}first order reaction network $\mathcal N$} is called weakly reversible if for any {\color{black} reaction} $S_i \rightarrow S_j$ with $i\not= j$, there exists a {\color{black} chain of reactions} $S_j \equiv S_{j_1} \rightarrow S_{j_2}\rightarrow\ldots \rightarrow S_{j_r} \equiv S_i$ where $S_{j_1}, S_{j_2}, \ldots, S_{j_r}$ are other {\color{black}chemical substances of $\mathcal N$}.  

\textcolor{black}{If a reaction network $\mathcal N$ is weakly reversible, then we also call the corresponding directed graph $G$  \emph{weakly reversible}.}
\end{definition}

\begin{definition}[Strongly connected components of a directed graph]\label{stronglyconnected}\hfill\\
A subgraph $H\subset G$ of a directed graph $G$ is called a strongly connected component  if for any two nodes $S_i, S_j$ in $H$, we can find a path from $S_i$ to $S_j$ of the form $S_i \rightarrow S_{i_1} \rightarrow \ldots \rightarrow S_{i_r} \rightarrow S_j$ with all $S_{i_1}, S_{i_2}, \ldots, S_{i_r}$ belonging to $H$. 

\textcolor{black}{We call a first order reaction network $\mathcal N$ strongly connected when its corresponding graph $G$ is \emph{strongly connected}.}
\end{definition}

\begin{remark}[Partition of weakly reversible first order reaction networks $\mathcal{N}$ into disjoint strongly connected components/subnetworks]\hfil\label{Partitions}\\	
\textcolor{black}{
Firstly, it follows directly from Definition \ref{linkageclass} that any first order reaction network $\mathcal{N}$
can be uniquely partitioned into a pairwise disjoint union of linkage classes and each 
linkage class $\mathcal{L}$ constitutes a connected subnetwork $\mathcal{N}_{\mathcal{L}}$.
In particular, for a weakly reversible first order reaction network $\mathcal N$, each linkage class $\mathcal{L}$ 
forms a connected weakly reversible subnetwork $\mathcal{N}_{\mathcal{L}}$
and it is straightforward to show that  the directed graph corresponding to $\mathcal{N}_{\mathcal{L}}$ is strongly connected according to Definition \ref{stronglyconnected}. (Consider that for all reactions being part of the connection between $S_i, S_j\in \mathcal{N}_{\mathcal{L}}$, the weak reversibility implies the existence of a returning chain of reactions. Thus, there exist  chains of reactions connecting $S_i$ to $S_j$ and vice versa.)  
}
black
\textcolor{black}{
Secondly, any directed graph $G$ can be partitioned into a pairwise disjoint union of strongly connected components, 
all of which are weakly reversible according to Definition \ref{weaklyreversible}. Note that these strongly connected components 
can still be connected via ``non-weakly-blackreversible" reactions (see e.g. Figure \ref{NonReversibleReaction}).
Therefore, for general directed graphs, multiple strongly connected components may constitute one linkage class. 
However, if the directed graph $G$ is additionally weakly reversible, then each strongly connected component has to constitute exactly one 
linkage class since otherwise we have already seen that weakly reversible subnetworks $\mathcal{N}_{\mathcal{L}}$ corresponding to one linkage class $\mathcal{L}$ are strongly connected. 
}

\textcolor{black}{
Thus, for weakly reversible first order reaction networks $\mathcal{N}$, the partition of linkage classes  
is identical to the partition of strongly connected components of the corresponding directed graphs.
}

{\color{black} Therefore, with a marginal abuse of notation, we will use the terminology ``strongly connected component" or ``strongly connected subnetwork" both for such a connected weakly reversible first order reaction subnetwork $\mathcal{N}_{\mathcal{L}}$ and its corresponding 
strongly connected subgraph/component.} 

\end{remark}

\begin{remark}[\textcolor{black}{Linkage classes of first order  reaction networks can be treated independently}]\hfil\label{StronglyConnected}\\	
For first order reaction networks, each node represents exactly one substance.
\textcolor{black}{
Thus, any linkage class of a first order reaction network 
can be treated independently from the others. 
In particular, all the strongly connected components of 
a weakly reversible first order reaction network can be treated independently since these subnetworks form different linkage classes. 
}

For higher order reaction networks, 
\textcolor{black}{where the nodes of the corresponding graphs are so-called complexes consisting of multiple substances,}
this is not necessarily true since 
one substance might need to be represented by different nodes. 
%
\end{remark}

\medskip
Because of Remarks \ref{Partitions} and \ref{StronglyConnected}, we will consider in Section \ref{weakly}
\textcolor{black}{weakly reversible first order networks partitioned into 
strongly 
connected first order reaction subnetworks $\mathcal{N}_{\mathcal{L}}$, 
and each strongly connected component $\mathcal{N}_{\mathcal{L}}$ can (w.l.o.g) be treated independently.}
In Section \ref{non-weakly}, we will consider (w.l.o.g) connected reaction networks $\mathcal{N}$ consisting of  one linkage class, yet we shall not assume {\color{black} weak reversibility. Hence the corresponding directed graphs are not strongly connected and may consists of multiple strongly connected components, but the underlying undirected graphs are connected (see e.g. Figure \ref{NonReversibleReaction})}. 

\medskip
%
%

%

\begin{lemma}[Strongly connected networks, irreducible reaction matrices and complex balance equilibria]\label{Characteristic}\hfil\\
\textcolor{black}{
For any first order reaction network $\mathcal N$ the following statements are equivalent:}
\begin{itemize}
\item 
The first order reaction network $\mathcal{N}$ is \textcolor{black}{strongly connected}.
\item The corresponding reaction matrix $A$ of $\mathcal{N}$ is irreducible. 
\item 
The first order reaction network $\mathcal{N}$ is complex balanced and for 
any positive mass $M>0$ 
(as set by the conservation law \eqref{MassConservation}), and there exists of a unique, positive 
complex balance equilibrium $X_{\infty} = (u_{1,\infty}, u_{2,\infty}, \ldots, u_{N,\infty})>0$ of system \eqref{VectorSystem}, 
which satisfies 
\begin{equation}\label{Equilibrium}
		\begin{cases}
			AX_{\infty} = 0,\\
			\sum_{i=1}^{N}u_{i,\infty} = M>0.
		\end{cases}
\end{equation}
\end{itemize}

\end{lemma}
\begin{proof}
The equivalence of strong connectivity for first order networks and irreducibility of the reaction matrix $A$ follows e.g. from \cite[Definition 2.1, page 46]{Seneta} and \cite[Theorem 3.2, page 78]{Min88}. Next, 
the Perron-Frobenius theorem implies for any irreducible reaction matrix $A$ and any positive mass $\sum_{i=1}^{N}u_{i,\infty}=M>0$ the existence of a unique positive equilibrium, see e.g. \cite{Seneta,Per07} and Lemma \ref{UniqueEqui} below.
 This equilibrium satisfies $AX_{\infty} = 0$ and is thus a complex 
balance equilibrium according to Definition \ref{ComplexBalance}. Hence, the strongly connected first order reaction network $\mathcal{N}$ is complex balanced (independently of the value of $M$).
Finally, Lemma \ref{UniqueEqui} below implies that strongly connected first order reaction networks possessing unique positive equilibrium (for fixed $M>0$) have irreducible reaction matrices $A$. 
\end{proof}

\begin{remark}[Complex balanced higher order systems are necessarily weakly reversible]\hfill \\
For higher order reaction network, it holds only true that systems with complex balance equilibrium
are necessarily weakly reversible. Thus, weakly reversible systems constitute the more general class of reaction networks.    
\end{remark}

\begin{remark}\label{PDEsetting}
The equilibrium $X_{\infty}$ in \eqref{Equilibrium} is spatially homogeneous. 
Thus, it coincides with the equilibrium for the corresponding spatially homogeneous ODE system $X_t = AX$ of the reaction network given in Figure \ref{Reaction}. In \cite{Ar11} or \cite{SiMa}, the authors proved that $X(t) \longrightarrow X_{\infty}$ as $t\longrightarrow +\infty$. However, the method used in this paper cannot be  directly applied to prove the convergence to equilibrium for PDE system \eqref{VectorSystem}. 
\end{remark}

The first main result of this paper concerns the convergence to equilibrium for weakly reversible reaction networks of the form displayed in Figure \ref{Reaction}. 
Our method of proof applies the entropy method to prove explicit exponential convergence of solutions of system \eqref{VectorSystem} to the unique equilibrium.

As mentioned above, all previous results of explicit EED inequalities 
(see e.g. \cite{DeFe06,DeFe_Con,DeFe08,DeFe15,BaFeEv14, MiHaMa14})
considered reaction-diffusion systems satisfying a detailed balance condition.
\medskip

In the current paper, we shall show that 
the following quadratic relative entropy between any two solutions $X=(u_1,\ldots,u_N)$ and $Y=(v_1,\ldots,v_N)$
\begin{equation}\label{RelativeEntropy_Quad}
	\mathcal{E}(X|Y)(t) = \sum_{i=1}^{N}\int_{\Omega}\frac{|u_i|^2}{v_i}dx
\end{equation}
is an entropy functional, see Lemma \ref{ExplicitEnDiss} below, which is the first key result of this paper.

In particular, we can consider the special case $Y=X_{\infty}$ for such an entropy functional.  
By using the linearity of first order systems, it is then straightforward to check (by using \eqref{a_jj} and $AX_{\infty}=0$)
that the quadratic relative entropy towards an equilibrium state $X_{\infty}$, i.e. 
\begin{equation}\label{RelativeEntropy}
	\mathcal{E}(X-X_{\infty}|X_{\infty}) = \sum_{i=1}^{N}\int_{\Omega}\frac{|u_i - u_{i,\infty}|^2}{u_{i,\infty}}dx
\end{equation}
is equally an entropy functional, which decays monotone in time according to 
the following explicit form of the entropy dissipation functional $\frac{d}{dt}\mathcal{E}(X- X_{\infty}|X_{\infty})=-\mathcal{D}(X-X_{\infty}|X_{\infty})$:
\begin{equation}\label{EntropyDissipation}
\begin{aligned}
\mathcal{D}(X-X_{\infty}|X_{\infty}) 
&= 2\sum_{i=1}^{N}d_i\int_{\Omega}\frac{|\nabla (u_i- u_{i,\infty})|^2}{u_{i,\infty}}dx\\
&\quad\, + \sum_{i,j=1;i<j}^{N}(a_{ji}u_{i,\infty} + a_{ij}u_{j,\infty})\int_{\Omega}\left(\frac{u_i- u_{i,\infty}}{u_{i,\infty}} - \frac{u_j- u_{j,\infty}}{u_{j,\infty}}\right)^2dx\ge0\\
&= 2\sum_{i=1}^{N}d_i\int_{\Omega}\frac{|\nabla u_i|^2}{u_{i,\infty}}dx
 + \sum_{i,j=1;i<j}^{N}(a_{ji}u_{i,\infty} + a_{ij}u_{j,\infty})\int_{\Omega}\left(\frac{u_i}{u_{i,\infty}} - \frac{u_j}{u_{j,\infty}}\right)^2dx\\ &= \mathcal{D}(X|X_{\infty})= - \frac{d}{dt}\mathcal{E}(X|X_{\infty}) \ge 0.\\
\end{aligned}
\end{equation}

The dissipative structure of the quadratic relative entropy towards equilibrium \eqref{RelativeEntropy} is a special cases of generalised relative entropies discussed e.g. in \cite[Chapter 6]{Per07}. 
The entropy functional \eqref{RelativeEntropy_Quad}, i.e. the observation of the dissipativeness  
of the relative entropy between any two solutions, is however related to 
a general property of linear Markow processes, which was recently shown in \cite{FJ16}.

With the help of the explicit form of entropy dissipation \eqref{EntropyDissipation}, we are able to show (in Lemma \ref{EEDEstimate}
below) an entropy-entropy dissipation inequality of the form
\begin{equation}\label{EEDa}
	\mathcal{D}(X-X_{\infty}|X_{\infty}) \geq \lambda\,\mathcal{E}(X-X_{\infty}|X_{\infty}),
\end{equation}
where $\lambda>0$ is an explicitly computable constant. 
Once the EED inequality \eqref{EEDa} is proven, the statement of the first main theorem follows from a standard 
Gronwall argument, see Section \ref{weakly} below:
\begin{theorem}[Exponential equilibration of weakly reversible first order reaction networks]\hfill\label{FirstResult}\\
\textcolor{black}{
Given a weakly reversible first order reaction network partitioned into linkage classes. 
Consider (w.l.o.g.) any corresponding strongly connected subnetwork $\mathcal{N}_{\mathcal{L}}$.
Assume for $\mathcal{N}_{\mathcal{L}}$} that 
the diffusion coefficients $d_i$ are positive for all $i = 1,2,\ldots, N$, and the initial mass $M$ is positive. 

Then, the unique global solution to initial-boundary problem \eqref{VectorSystem} converges exponentially to the unique positive equilibrium $X_{\infty} = (u_{1,\infty}, u_{2,\infty}, \ldots, u_{N,\infty})$, i.e.
\begin{equation*}
\sum_{i=1}^{N}\int_{\Omega}\frac{|u_i(t) - u_{i,\infty}|^2}{u_{i,\infty}}dx \leq e^{-\lambda t}\sum_{i=1}^{N}\int_{\Omega}\frac{|u_{i,0} - u_{i,\infty}|^2}{u_{i,\infty}}dx,
\end{equation*}
where the constant $\lambda >0$ {\color{black}depends explicitly on the reaction matrix $A$, the domain $\Omega$, the diffusion matrix $D$ and the initial mass $M$.}
\end{theorem}

\begin{remark}[Lyapunov functionals for ODE systems]\hfil\label{Lj}\\
{For ODE systems, Lyapunov functionals have been mainly considered in the analysis of nonlinear ODE systems. 
Moreover, for nonlinear ODE systems, $L^1$-type Lyapunov functionals
are most commonly used in the study of the large-time-behaviour.}
For reaction-diffusion systems, however, $L^1$-functionals are not useful for the entropy method and proving explicit convergence to equilibrium, since they do not measure the spatial diffusion in an exploitable way.    

We also remark, that while logarithmic relative entropy functionals of the form 
\begin{equation}\label{Entropy-Like}
V_{X_{\infty}}(X)(t) = \sum_{i=1}^{N}\left(u_i(\ln u_i - \ln u_{i,\infty} - 1) + u_{i,\infty}\right)
\end{equation}
were known to constitute a monotone decaying Lyapunov functional for complex balanced ODE reaction networks (see e.g. \cite{HoJa72, Man,SiMa}), up to our knowledge and somewhat surprisingly,
no explicit expression of the {\it entropy dissipation} $-dV/dt$ in complex balanced systems has been derived so far. 

We also refer the reader to e.g. \cite{MiSi} for the stability of some mass action law reaction-diffusion systems, where the author used techniques of $\omega$-limit sets along with the monotonicity of $L^1$-type Lyapunov functional. 

Our results in this paper are significantly stronger in the sense that we show, by using the entropy method, the exponential convergence to equilibrium with computable rates. 

In addition and in comparison to $\omega$-limit techniques, the entropy method has also the major advantage of relying on functional inequalities rather than on specific estimates of solutions to a given system. Having such functional entropy entropy-dissipation inequalities once and for all established makes the entropy method robust with respect to model variations and generalisations. 

As example, it is the intrinsic  robustness of the entropy method, which makes it possibly to also apply to non weakly reversible reaction networks, see Theorems \ref{SourceTransmission} and \ref{Target} below.   
\end{remark}

The assumption on the positivity of all diffusion coefficients in Theorem \ref{FirstResult} is not necessary as such. As already shown in e.g.  \cite{DeFe_Con, BaFeEv14}, the combined effect of diffusion of a specie 
and its weakly reversible reaction with other (possibly non-diffusive) 
species will lead to a indirect ``diffusion-effect" on the latter specie. 
This indirect diffusion-effect can also be measured in terms of functional inequalities. Hence the exponential convergence to equilibrium still holds for systems with partial degenerate diffusion. 

\textcolor{black}{
Note that the indirect ``diffusion transfer'' and the convergence results of this paper resembles 
to some degree the framework of hypocoercivity for evolution equations like linear kinetic Fokker-Planck equations, 
see e.g. \cite{Vil09,DMS,AAS}. However, while hypocoercivity 
typically requires the use of suitably constructed Lyapunov functionals, the indirect ``diffusion-effect" can be entirely express in functional inequalities linking the relative entropy and the associated entropy dissipation functional. 
The entropy method present in this paper proves convergence to equilibrium essentially regardless of full- or degenerate diffusion matrices.}

The exponential convergence for weakly reversible systems \eqref{VectorSystem} with degenerate diffusion is stated in the 
following Theorem \ref{DegenerateDiffusion} to be proved in Section \ref{weakly} below:
\begin{theorem}[Equilibration of linear networks with degenerate diffusion]\hfill\label{DegenerateDiffusion}\\
\textcolor{black}{
Given a weakly reversible first order reaction network partitioned into linkage classes. 
Consider (w.l.o.g.) any corresponding strongly connected subnetwork $\mathcal{N}_{\mathcal{L}}$.
Assume that the initial mass $M$ is positive for $\mathcal{N}_{\mathcal{L}}$}. 
Moreover, assume that at least one diffusion coefficient $d_i$ is positive for some $i = 1, 2, \ldots, N$. 

Then, the solution to \eqref{VectorSystem} converges exponentially fast to the unique positive equilibrium $X_{\infty} = (u_{1,\infty}, u_{2,\infty}, \ldots, u_{N,\infty})$:	\begin{equation*}
		\sum_{i=1}^{N}\int_{\Omega}\frac{|u_i(t) - u_{i,\infty}|^2}{u_{i,\infty}}dx \leq e^{-\lambda' t}\sum_{i=1}^{N}\int_{\Omega}\frac{|u_{i,0} - u_{i,\infty}|^2}{u_{i,\infty}}dx
	\end{equation*}
with a computable rate $\lambda'>0$, {\color{black}which depends explicitly on $A$, $\Omega$, $D$ and $M$}.
\end{theorem}
\begin{remark}[Same results of linear ODE reaction networks]\hfil\\
We remark that our approach can of course be adapted to equally apply  to linear ODE reaction networks by eliminating the terms and calculations  concerning spatial diffusion. Thus, all the results of this paper
hold equally for such linear ODE systems.  
\end{remark}
\medskip

As the second main result of this manuscript, we shall derive an entropy approach and prove convergence to equilibrium for reaction networks as in Figure \ref{Reaction}, for which {\it the weak reversibility assumption does not hold}. 
For first order reaction networks, this implies that the system is not complex balanced, or in other words, that equilibria are not necessarily positive.

Due to the lack of positivity of equilibria, it follows immediately that the relative entropy used for weakly reversible systems is not directly applicable. In the following we proposed a modified entropy approach. At first, it is necessary to understand the structure of non weakly reversible reaction networks.

We state here the necessary terminology and the main ideas.
Since for any non weakly reversible linkage class, the associated directed graph $G$ is connected (which means that the underlying undirected version of $G$ is a connected graph) but not strongly connected, $G$ consists of $r\geq 2$ strongly connected components, which we denote by $C_1, C_2, \ldots, C_r$. Then, we can construct a {\it directed acyclic graph} $G^C$, i.e. $G^C$ is a directed graph with no directed cycles as follows:
\begin{itemize}
	\item[a)] $G^C$ has as nodes the $r$ strongly connected components $C_1,C_2, \ldots, C_r$,
	\item[b)] for two nodes $C_i$ and $C_j$ of $G^C$, if there exists a reaction $C_i\ni S_{k} \xrightarrow{a_{\ell k}}S_{\ell} \in C_j$ with $a_{\ell k}>0$, then there exists also the edge $C_i \rightarrow C_j$ on $G^C$.
\end{itemize}
Due to the structure of $G^C$, its nodes, or equivalently the strongly connected components of $G$, can be labeled as one of the following three types:
\begin{itemize}
	\item A strongly connected component $C_i$ is called a {\it source component} if there is no in-flow to $C_i$, i.e. there does not exist an edge $S_k\rightarrow S_j$ where $S_k\not\in C_i$ and $S_j\in C_i$.
	\item A strongly connected component $C_i$ is called a {\it target component} if there is no out-flow from $C_i$, i.e. there does not exist and edge $S_k \rightarrow S_j$ where $S_k\in C_i$ and $S_j\not\in C_i$.
	\item If $C_i$ is neither a source component nor a target component, then we call $C_i$ a {\it transmission component}.
\end{itemize}

\begin{example}
			Consider the reaction network in Figure \ref{NonReversibleReaction}.
			\begin{figure}[htp]
			\begin{center}\scalebox{1}[1]{
			\begin{tikzpicture}
			\node (a) {$S_{1}$} node (b) at (0,2) {$S_{2}$} node (c) at (2,2) {$S_3$}  node (d) at (2,0) {$S_4$}  node (e) at (4,0) {$S_5$}   node(f) at(4,2){$S_6$};
			\draw[arrows=->] ([xshift=-0.5mm, yshift=0.5mm]a.north) -- node [left] {\scalebox{.8}[.8]{$a_{21}$}} ([xshift=-0.5mm, yshift=-0.5mm]b.south);
			\draw[arrows=->] ([xshift =0.5mm,yshift=-0.5mm]b.south) -- node [right] {\scalebox{.8}[.8]{$a_{12}$}} ([xshift=0.5mm,yshift=0.5mm]a.north);
			\draw[arrows=->] ([xshift =0.5mm]b.east) -- node [above] {\scalebox{.8}[.8]{$a_{32}$}} ([xshift=-0.5mm]c.west);
			\draw[arrows=->] ([xshift =0.5mm,yshift=1mm]a) -- node [right] {\scalebox{.8}[.8]{$a_{31}$}} ([xshift=-0.5mm,yshift=-1mm]c);
			\draw[arrows=->] ([xshift =0.5mm]a) -- node [below] {\scalebox{.8}[.8]{$a_{41}$}} ([xshift=-0.5mm]d);
			\draw[arrows=->] ([xshift =0.5mm,yshift=-0.5mm]d.east) -- node [below] {\scalebox{.8}[.8]{$a_{54}$}} ([xshift=-0.5mm,yshift=-0.5mm]e.west);
			\draw[arrows=->] ([xshift =-0.5mm,yshift=.5mm]e.west) -- node [above] {\scalebox{.8}[.8]{$a_{45}$}} ([xshift=0.5mm,yshift=.5mm]d.east);
			\draw[arrows=->] ([xshift =0.5mm]c) -- node [above] {\scalebox{.8}[.8]{$a_{63}$}} ([xshift=-0.5mm]f);
			\draw[arrows=->] ([yshift =-0.5mm]c) -- node [right] {\scalebox{.8}[.8]{$a_{43}$}} ([yshift=0.5mm]d);
			\end{tikzpicture}} \caption{A non-weakly reversible reaction network consisting of four strongly connected components}\label{NonReversibleReaction} 
			\end{center}
			\end{figure}
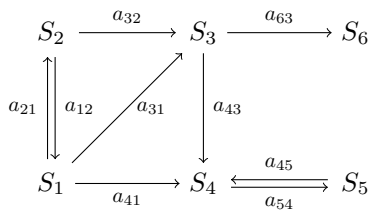
The depicted network has $4$ strongly connected components $C_1 = \{S_1, S_2\},\, C_2 = \{S_3\},\, C_3 = \{S_4, S_5\},\, C_4 = \{S_6\}$, where $C_1$ is a source component, $C_2$ is a transmission component and $C_3, C_4$ are target components.
\end{example}

By definition, each of the three types of strongly connected components 
is subject to a different dynamic, which can be written as follows: 
Let $C_i$ be a strongly connected component and denote by $X_i$
the concentrations within $C_i$. Moreover, denote by $A_i$ the reaction matrix formed by all reactions within the component $C_i$.
Then, we have
\begin{itemize}
\item for a source component $C_i$:
\begin{equation*}
\partial_t{X}_{i} - D_{i}\Delta X_{i} = A_iX_{i} + \mathcal{F}_i^{out},
\end{equation*}
where $\mathcal{F}_i^{out}$ summarises the out-flow from the source component $C_i$. 
\item for a target component $C_i$:
\begin{equation*}
\partial_t{X}_{i} - D_{i}\Delta X_{i} = \mathcal{F}_i^{in} + A_iX_{i},
\end{equation*}
where $\mathcal{F}_i^{in}$ summarises the in-flow into the target component $C_i$. 
\item for a transmission component $C_i$:
		\begin{equation*}
			\partial_t{X}_{i} - D_{i}\Delta X_{i} = \mathcal{F}_i^{in} + A_iX_{i} + \mathcal{F}_i^{out},
		\end{equation*}
where $\mathcal{F}_i^{in}$, $\mathcal{F}_i^{out}$ are the 
in/out-flow of the transmission component $C_i$.		
\end{itemize}

In the dynamics of transmission and target components, the in-flow $\mathcal{F}_i^{in}$ depends only on species which do not belong to $C_i$, so that $\mathcal{F}_i^{in}$ can be treated as an external source for the system for $C_i$. However, it may happen that $\mathcal{F}_i^{in}$ contains inflow from species whose behaviour is not a-priori known. 

For acyclic graphs $G^C$, however, it is possible to avoid these difficulties, 
since the  {\it topological order} of acyclic graphs allows to re-order the $r$ strongly connected components $C_1,C_2,\ldots, C_r$ in such a way  that for every edge $C_i \rightarrow C_j$ of $G^C$ it holds that $i<j$.
This permits to study the dynamics of all components $C_i$ sequentially 
according to the topological order and, when at times considering a transmission component (or later a target component) $C_i$, the required in-flow $\mathcal{F}_i^{in}$ contains only species whose behaviour 
is already known.
\medskip

Due to the structure of the network, it is expected that species belonging to source or transmission components are subsequently losing mass such that
the concentrations decay to zero in the large-time behaviour as time goes to infinity. In contrast, the species belonging to a target component 
converge to an equilibrium state, which is determined by the reactions within this component and by the mass "injected" from other components. 

Since the source and transmission components do not converge to positive equilibria, the relative entropy method used for weakly reversible networks 
is directly not applicable. Instead, for each component $C_i$ we will modify the entropy method by introducing an \emph{artificial equilibrium state with normalised mass}, which balances the reaction within $C_i$. The artificial equilibrium will allow us to consider a quadratic functional, which is similar to the relative entropy in weakly reversible networks and  which can be proved to decay exponentially to zero. This result is stated in the following Theorem:
\begin{theorem}[Exponential decay to zero of source and transmission components]\hfill\label{SourceTransmission}\\
\textcolor{black}{Given an arbitrary first order reaction network partitioned into linkage classes and consider (w.l.o.g.) any corresponding connected subnetwork $\mathcal{N}_{\mathcal{L}}$. 
Assume for $\mathcal{N}_{\mathcal{L}}$ that all diffusion coefficients $d_i$ are positive.} 

Then, for each $C_i$ being a source or a transmission component of $\mathcal{N}_{\mathcal{L}}$, there exist constants $K_i>0$ and $\lambda_i>0$ {\color{black}depending explicitly on $A_i$ and $\Omega$} such that, for any specie $S_{\ell} \in C_i$, the concentration $u_{\ell}$ of $S_{\ell}$ decays  exponentially to zero, i.e.
\begin{equation*}
\|u_{\ell}(t,\cdot)\|_{L^2(\Omega)}^2 \leq K_ie^{-\lambda_i t}, \qquad \text{ for all }\ t>0.
\end{equation*}
\end{theorem}


For a target component $C_i$, due to the in-flow $\mathcal{F}_i^{in}$, the total mass of $C_i$ is not conserved but increasing. Hence, $C_i$ does not possesses an equilibrium as weakly reversible networks, which is  explicitly  
given in terms of the reaction rates and the conserved initial total mass. 

However, since each target component is strongly connected and thus a weakly reversible reaction network with mass influx, there still exists a unique, 
positive equilibrium of $C_i$ denoted by $X_{i,\infty}$, which balances the reactions within $C_i$ and has a total mass, which is the sum of the total initial mass of $C_i$ and the total "injected mass" from the other components via the in-flow $\mathcal{F}_i^{in}$ (see Lemma \ref{FinalStateTarget}). 
We emphasis that in general the injected mass is not given explicitly but depends on the time evolution of all the influencing species higher up with respect to the topological order of the graph $G^C$.

Since the equilibrium $X_{i,\infty}$ is positive, we can use again a relative entropy functional to prove the convergence of the species belonging to a target component to their corresponding equilibrium states.
\begin{theorem}[Exponential convergence for target components]\hfill\label{Target}\\
\textcolor{black}{Given an arbitrary first order reaction network partitioned into linkage classes and consider (w.l.o.g.) any corresponding connected subnetwork $\mathcal{N}_{\mathcal{L}}$.}
Assume for $\mathcal{N}_{\mathcal{L}}$ that all diffusion coefficients $d_i$ are positive.

Then, for all target components $C_i = \{S_{i_1}, S_{i_2}, \ldots, S_{i_{N_i}}\}$ of $\mathcal{N}_{\mathcal{L}}$, where $N_i$ is the number of species belonging to $C_i$, there exists a unique positive equilibrium state $X_{i,\infty} = (u_{i_1,\infty}, \ldots, u_{i_{N_i},\infty})$ and 
the concentrations $u_{i_\ell}$ of $S_{i_\ell}$ converges exponentially to the corresponding equilibrium value
\begin{equation*}
\|u_{i_\ell}(t) - u_{i_\ell,\infty}\|_{L^2(\Omega)}^2 \leq K_ie^{-\lambda_i t}, \qquad \text{ for all } t>0,
\end{equation*}
{\color{black}with the constants $K_i>0$ and $\lambda_i >0$ depending explicitly 
on $A_i$, $\Omega$ and $D_i$ and on the equilibrium state $X_{i,\infty}$. }
\end{theorem}
\textcolor{black}{
\begin{remark}
Note that by Lemma \ref{FinalStateTarget}, the equilibrium state $X_{i,\infty}$ depends explicitly on the mass injected into the target component $C_i$, 
but that injected mass itself depends non-explicitly on the initial data and on the history of the reaction-diffusion network.
\end{remark}}
\begin{remark}
We remark that in the same way as Theorem \ref{DegenerateDiffusion} generalises Theorem \ref{FirstResult} to allow for 
degenerate diffusion matrices, it is equally possible to generalise Theorems \ref{SourceTransmission} and \ref{Target}
in the sense that it is sufficient to assume that for each target component there is at least one diffusion coefficient is positive. 
In particular, the proof of Theorem \ref{SourceTransmission} holds independently from the entries of a non-negative diffusion matrices $D_i$.
\end{remark}

\medskip
\noindent\underline{Outline:}
The rest of the paper is organised as follows: In Section 2, we present the entropy method for weakly reversible networks and prove exponential convergence to the positive equilibrium. Non weakly reversible networks will be investigated in the Section 3. By using the structure of the underlying graphs, we are able completely resolve the large-time behaviour of all species belonging to such first order networks.

{We also remark that all constants in this manuscript are explicit in the sense that they are derived in constructive ways. However, since these constants are not optimal, we will denote them by using generic letters like $K_i$ or $\lambda_i$, etc.
The issue of optimal rates and constants for the convergence is subtle, and can be investigated in future works.}
\medskip

\noindent\underline{Notation:} We shall 
use the shortcut $\overline{f}=\iO f(x) \,dx$, 
whenever $|\Omega|=1$, and $\|\cdot\|$ for the usual norm in $L^2(\Omega)$, i.e.
\begin{equation*}
			\|f\|^2 = \int_{\Omega}|f(x)|^2dx.
\end{equation*}

\section{Strongly connected first order networks}\label{weakly}
\textcolor{black}{
In this section, we consider strongly connected first order reaction networks $\mathcal{N}$, for which the  
associated directed graph 
is strongly connected.
This is w.l.o.g. by Remarks \ref{Partitions} and \ref{StronglyConnected}, since any weakly reversible first order reaction network 
can be partitioned into disjoint strongly connected components/subnetworks, which can be treated independently. 
}

Moreover, we recall that $\Omega\subset\mathbb{R}^n$ is a bounded domain with smooth boundary $\partial\Omega$ (say $\partial\Omega\in C^{2+\alpha}$) and normalised volume $|\Omega| = 1$ (w.l.o.g. by rescaling).
Finally, we recall the system \eqref{VectorSystem}
\begin{equation}\label{VectorSystem_ReCall}
	\begin{cases}
		\partial_tX - D\Delta X = AX, &\quad x\in\Omega, \quad t>0,\\
		\partial_{\nu}X = 0, &\quad x\in\partial\Omega, \quad t>0,\\
		X(x,0) = X_0(x), &\quad x\in\Omega,
	\end{cases}
\end{equation}
where $X = [u_1, u_2, \ldots, u_N]^T$ denotes the vector of concentrations, the vector $X_0 = [u_{1,0}, u_{2,0}, \ldots, u_{N,0}]^T$ denotes the initial data, the diffusion matrix $D = \text{diag}(d_1,d_2,\ldots, d_N)$ and the reaction matrix $A = (a_{ij}) \in \mathbb{R}^{N\times N}$ satisfies
\begin{equation}\label{a_jj_Recall}
\begin{cases}
	a_{ij} \geq 0, &\qquad \text{for all } i\not=j, \quad i,j =1,2,\ldots,N,\\ 
	a_{jj} = -\sum_{i=1,i\not= j}^{N}a_{ij}, &\qquad \text{for all } j =1,2,\ldots,N.
\end{cases}
\end{equation}
Moreover, since $\mathcal{N}_{\mathcal{L}}$ is strongly connected, we know that the reaction matrix $A$ is irreducible, see Lemma \ref{Characteristic}.
For the linear system \eqref{VectorSystem_ReCall}, the existence of a global unique solution 
follows by standard arguments, see e.g. \cite{Smo,Rot}:
\begin{theorem}[Global well-posedness of linear reaction-diffusion networks]\hfil\\
For all given initial data $X_0 \in (L^2(\Omega))^N$, there exists a unique  solution $X\in C([0,T];(L^2(\Omega))^N)\cap L^2(0,T;(H^1(\Omega))^N)$ for all $T>0$. Moreover, if $X_0 \geq 0$ then $X(t) \geq 0$ for all $t>0$. Finally, the solutions to \eqref{VectorSystem_ReCall} conserve the total mass \eqref{MassConservation} for all $t>0$:
\begin{equation}\label{MassConservation_ReCall}
	\sum_{i=1}^{N}\int_{\Omega}u_i(x,t)dx = \sum_{i=1}^{N}\int_{\Omega}u_{i,0}(x)dx =: M >0,
\end{equation}
where the initial mass $M$ is assumed positive.
\end{theorem}

\medskip


Lemma \ref{Characteristic} stated the equivalence to weak reversibility 
first order reaction networks and irreducibility of the reaction matrices $A$, which follows e.g. from \cite[Definition 2.1, page 46]{Seneta} and \cite[Theorem 3.2, page 78]{Min88}.
Moreover, Lemma \ref{Characteristic} stated the existence of a unique positive complex balance equilibrium to \eqref{VectorSystem_ReCall} 
for any given positive initial mass $M>0$.   
Concerning the proof of this part of Lemma \ref{Characteristic}, it remains to show the following

\begin{lemma}[Unique positive equilibria for \textcolor{black}{strongly connected} networks with fixed mass $M$]\label{UniqueEqui}\hfil\\
The {\color{black} first order} reaction network $\mathcal{N}$ is \textcolor{black}{strongly connected} if and only if 
the system \eqref{VectorSystem_ReCall} admits a unique positive equilibrium for any fixed positive mass $M>0$.
\end{lemma}
\begin{proof}
{\it Sufficiency:} 
Assume that $\mathcal{N}$ is \textcolor{black}{strongly connected}. Thanks to the first equivalency in Lemma \ref{Characteristic}, the reaction matrix $A$ is irreducible. Moreover, for large enough $\alpha>0$, we have that $A+ \alpha E$ is nonnegative in the sense that all of its elements are nonnegative. We can then apply an extension of the Perron-Frobenius theorem, see e.g. \cite[Theorem 2.6, page 46]{Seneta} or \cite[Chapter 6.3.1]{Per07}, to obtain the existence of a unique positive equilibrium, i.e. a positive right zero-eigenvector $X_{\infty} = (u_{1,\infty}, u_{2,\infty}, \ldots, u_{N,\infty})>0$ satisfying $AX_{\infty} = 0$ such that $\sum_{i=1}^{N}u_{i,\infty} = M>0$.
	
	{\it Necessity:} {{Now assume that \eqref{VectorSystem_ReCall} has a unique positive equilibrium $X_{\infty}$. Since $AX_{\infty} = 0$ and $X_{\infty}$ is uniquely determined by the mass conservation, we obtain that $\dim(\ker A) = 1$.
		
	By using a contradiction argument, we assume that $\mathcal{N}$ is not \textcolor{black}{strongly connected}, 
	then the reaction matrix $A$ is reducible, i.e.
	\begin{equation*}
		A = P^T\begin{pmatrix}
			B&0\\ C&D
		\end{pmatrix}P
	\end{equation*}
	for some permutation matrix $P$, in which $D$ is irreducible. Choose $d$ to be an eigenvector of $D$ corresponding to zero eigenvalue. Then
	\begin{equation*}
		AP^T\begin{pmatrix}
		0\\d
		\end{pmatrix} = P^T\begin{pmatrix}B&0\\C&D\end{pmatrix}\begin{pmatrix}
				0\\d
				\end{pmatrix} = P^T\begin{pmatrix}0\\Dd\end{pmatrix} = \begin{pmatrix}
				0\\0
				\end{pmatrix}
	\end{equation*}
	which means that $P^T\begin{pmatrix}
	0\\d
	\end{pmatrix}$ is an eigenvector of $A$ corresponding to zero eigenvalue. Since $X_{\infty}$ is strictly positive, $P^T\begin{pmatrix}
		0\\d
		\end{pmatrix}$ and $X_{\infty}$ are linear independent, which leads to a contradiction with $\mathrm{dim}(\mathrm{ker} A) = 1$.}}
%
\end{proof}

In the following, we will use the entropy method to study the trend to equilibrium. More precisely, for two trajectories $X = (u_1, u_2, \ldots, u_N)$ and $Y = (v_1, v_2, \ldots, v_N)$ to \eqref{VectorSystem_ReCall}, where $Y(t)$ has non-zero components for all times $t>0$, we consider the following quadratic relative entropy functional
\begin{equation}\label{RelativeEntropy_Recall}
	\mathcal{E}(X|Y)(t) = \sum_{i=1}^{N}\int_{\Omega}\frac{|u_i|^2}{v_i}dx.
\end{equation}

The following key Lemma \ref{ExplicitEnDiss} provides an 
explicit expression of the entropy dissipation associated to \eqref{RelativeEntropy_Recall}: 
\begin{lemma}[Relative entropy dissipation functional]\label{ExplicitEnDiss}\hfil\\
Assume that $v_i(t) \not= 0$ for all $i = 1,2,\ldots, N$ and $t>0$. Then,  we have
\begin{align*}
\mathcal{D}(X|Y) = -\frac{d}{dt}\mathcal{E}(X|Y)= 2\sum_{i=1}^{N}d_i\int_{\Omega}v_i\left|\nabla\Bigl(\frac{u_i}{v_i}\Bigr)\right|^2dx+ \!\!\sum_{i,j=1; i<j}^{N}\int_{\Omega}(a_{ij}v_j + a_{ji}v_i)\biggl(\frac{u_i}{v_i} - \frac{u_j}{v_j}\biggr)^{\!2}dx.
\end{align*}
\end{lemma}
\begin{proof}
For convenience we recall that
$$
\partial_tu_i - d_i\Delta{u_i} = \sum_{j=1}^{N}a_{ij}u_j \qquad \text{ and  } \qquad \partial_tv_i - d_i\Delta{v_i} = \sum_{j=1}^{N}a_{ij}v_j,
$$
for all $i=1,\ldots, N$. 
Hence, we compute
\begin{align}
\frac{d}{dt}\mathcal{E}(X|Y) 
&= \sum_{i=1}^{N}\int_{\Omega}\left[2\frac{u_i}{v_i}\partial_tu_i - \frac{u_i^2}{v_i^2}\partial_tv_i\right]dx\nonumber \\
&= \sum_{i=1}^{N}\int_{\Omega}\left[2\frac{u_i}{v_i}\biggl(d_i\Delta{u_i} + \sum_{j=1}^{N}a_{ij}u_j\biggr) - \frac{u_i^2}{v_i^2}\biggl(d_i\Delta{v_i} + \sum_{j=1}^{N}a_{ij}v_j\biggr)\right]dx\nonumber \\
&= \sum_{i=1}^{N}\int_{\Omega}\left(2d_i\frac{u_i}{v_i}\Delta{u_i} - d_i\frac{u_i^2}{v_i^2}\Delta{v_i}\right)dx + \sum_{i=1}^{N}\int_{\Omega}\biggl(2\frac{u_i}{v_i}\sum_{j=1}^{N}a_{ij}u_j - \frac{u_i^2}{v_i^2}\sum_{j=1}^{N}a_{ij}v_j\biggr)dx\nonumber \\
			&=: \sum_{i=1}^{N}\int_{\Omega}J^{(i)}_Ddx + \int_{\Omega}\sum_{i=1}^{N}J^{(i)}_Rdx\nonumber \\
			&=: \mathcal{I}_D + \mathcal{I}_R.\label{e1}
\end{align}
	Using integration by parts, we have
\begin{align}
\int_{\Omega}J^{(i)}_{D}dx &= \int_{\Omega}\left(2d_i\frac{u_i}{v_i}\Delta{u_i} - d_i\frac{u_i^2}{v_i^2}\Delta{v_i}\right)dx\nonumber \\
&= -2d_i\int_{\Omega}\left(\nabla\Bigl(\frac{u_i}{v_i}\Bigr)\nabla u_i - \frac{u_i}{v_i}\nabla\Bigl(\frac{u_i}{v_i}\Bigr)\nabla v_i\right)dx\nonumber \\
&= -2d_i\int_{\Omega}v_i\left|\nabla\Bigl(\frac{u_i}{v_i}\Bigr)\right|^2dx.
	 \label{e2} 
\end{align}
Thus,
\begin{equation}
 \label{e3} 
	\mathcal{I}_D = -2\sum_{i=1}^{N}d_i\int_{\Omega}v_i\left|\nabla\Bigl(\frac{u_i}{v_i}\Bigr)\right|^2dx.
\end{equation} 
For the reaction terms $\mathcal{I}_R$, we use
$
	a_{ii} = -\sum_{j=1,j\not=i}^{N}a_{ji}
$
to calculate
\begin{align}
J^{(i)}_R &= 2\frac{u_i}{v_i}\sum_{j=1}^{N}a_{ij}u_j - \frac{u_i^2}{v_i^2}\sum_{j=1}^{N}a_{ij}v_j\nonumber \\
&= 2\frac{u_i}{v_i}\biggl(\,\sum_{j=1,j\not=i}^{N}a_{ij}u_j + a_{ii}u_i\biggr) - \frac{u_i^2}{v_i^2}\biggl(\,\sum_{j=1,j\not=i}^{N}a_{ij}v_j + a_{ii}v_i\biggr)\nonumber \\
&= 2\frac{u_i}{v_i}\biggl(\,\sum_{j=1,j\not=i}^{N}a_{ij}u_j - u_i\sum_{j=1,j\not=i}^{N}a_{ji}\biggr) - \frac{u_i^2}{v_i^2}\biggl(\,\sum_{j=1,j\not=i}^{N}a_{ij}v_j - v_i\sum_{j=1,j\not=i}^{N}a_{ji}\biggr)\nonumber \\
&=\sum_{j=1,j\not=i}^{N}\left(2\frac{u_i}{v_i}(a_{ij}u_j - a_{ji}u_i) - \frac{u_i^2}{v_i^2}(a_{ij}v_j - a_{ji}v_i)\right).	 
\label{e4} 
\end{align}
Therefore,
\begin{align}
\mathcal{I}_R = \sum_{i=1}^{N}\int_{\Omega}J^{(i)}_Rdx&= \int_{\Omega}\sum_{i=1}^{N}\sum_{j=1,j\not=i}^{N}\left(2\frac{u_i}{v_i}(a_{ij}u_j - a_{ji}u_i) - \frac{u_i^2}{v_i^2}(a_{ij}v_j - a_{ji}v_i)\right)dx\\
&= \sum_{i,j=1;i<j}^{N}\int_{\Omega}\biggl[2\frac{u_i}{v_i}(a_{ij}u_j - a_{ji}u_i) - \frac{u_i^2}{v_i^2}(a_{ij}v_j - a_{ji}v_i)\nonumber \\
&\qquad\qquad\qquad+ 2\frac{u_j}{v_j}(a_{ji}u_i - a_{ij}u_j) - \frac{u_j^2}{v_j^2}(a_{ji}v_i - a_{ij}v_j)\biggr]dx\nonumber\\
&= \sum_{i,j=1;i<j}^{N}\int_{\Omega}\biggl[2(a_{ij}u_j - a_{ji}u_i)\left(\frac{u_i}{v_i}-\frac{u_j}{v_j}\right)- (a_{ij}v_j - a_{ji}v_i)\left(\frac{u_i^2}{v_i^2}-\frac{u_j^2}{v_j^2}\right)\biggr]dx\nonumber \\
&= \sum_{i,j=1;i<j}^{N}\int_{\Omega}\left(\frac{u_i}{v_i}-\frac{u_j}{v_j}\right)\left[2(a_{ij}u_j-a_{ji}u_i) -  (a_{ij}v_j - a_{ji}v_i)\left(\frac{u_i}{v_i}+\frac{u_j}{v_j}\right)\right]dx\nonumber \\
&= -\sum_{i,j=1;i<j}^{N}\int_{\Omega}(a_{ij}v_j + a_{ji}v_i)\left(\frac{u_i}{v_i}-\frac{u_j}{v_j}\right)^2dx.
 \label{e5} 
\end{align}
By combining \eqref{e1}, \eqref{e3} and \eqref{e5}, we obtain the result stated in the Lemma.
\end{proof}


In order to simplify the following calculations, we introduce the difference to the equilibrium  
$$
W:= (w_1,w_2,\ldots, w_N) = (u_1-u_{1,\infty}, u_{2} - u_{2,\infty}, \ldots, u_N - u_{N,\infty}) = X - X_{\infty},
$$
and remark that thanks to the linearity of the system, the difference $W$ is the
solution to \eqref{VectorSystem_ReCall} subject to the shifted initial data
$$
W(x,0) = X(x,0) - X_{\infty},\qquad \text{ for all } x\in\Omega.
$$
Note that the total initial mass corresponding to $W$ is zero, i.e. 
$$
M_{W} := \sum_{i=1}^{N}\iO{w_{i,0}}\,dx = \sum_{i=1}^{N}\int_{\Omega}(u_{i,0}(x) - u_{i,\infty})\,dx = 0,
$$
and that $W$ conserves the zero mass 
\begin{equation*}
	\sum_{i=1}^{N}\iO{w_i}(t,x)\,dx = 0, \qquad \text{ for all } t>0.
\end{equation*}
By using the relative entropy dissipation functional derived in  Lemma \ref{ExplicitEnDiss}, we have
\begin{equation*}
\mathcal{D}(W|X_{\infty}) = 2\sum_{i=1}^{N}\int_{\Omega}d_i\,\frac{|\nabla w_i|^2}{u_{i,\infty}}\,dx
+ \sum_{i,j=1;i<j}^{N}(a_{ij}u_{j,\infty} + a_{ji}u_{i,\infty})\int_{\Omega}\Bigl( \frac{w_i}{u_{i,\infty}} - \frac{w_j}{u_{j,\infty}}\Bigr)^{\!2}dx.
\end{equation*}

The following Lemma about entropy-entropy dissipation estimate is the main key to prove the convergence to equilibrium for \eqref{VectorSystem_ReCall}.
\begin{lemma}[Entropy-Entropy Dissipation Estimate]\label{EEDEstimate}\hfil\\
There exists an explicit constant $\lambda>0$ {\color{black}depending explicitly on the reaction matrix $A$, the domain $\Omega$, the diffusion matrix $D$ and the initial mass $M$} such that
	\begin{equation*}
		\mathcal{D}(W|X_{\infty}) \geq \lambda\,\mathcal{E}(W|X_{\infty}).
	\end{equation*}
\end{lemma}
\begin{proof}
We divide the proof in several steps:
	
\noindent {\bf Step 1.} (Additivity of the relative entropy w.r.t. spatial averages)\hfil\\
Straightforward calculation leads to 
\begin{align}
\mathcal{E}(W|X_{\infty}) &= \sum_{i=1}^{N}\int_{\Omega}\frac{|w_i|^2}{u_{i,\infty}}dx= \sum_{i=1}^{N}\int_{\Omega}\frac{|w_i - \overline{w_i}|^2}{u_{i,\infty}}dx + \sum_{i=1}^{N}\frac{|\overline{w_i}|^2}{u_{i,\infty}}\nonumber\\
&= \mathcal{E}(W-\overline{W}|X_{\infty}) + \mathcal{E}(\overline{W}|X_{\infty})
\label{e7} 
\end{align}
where we denote $\overline{W} = (\overline{w_1}, \overline{w_2}, \ldots, \overline{w_N})$ and we recall that $\overline{w_i}=\iO w_i\,dx$ for $i=1,\ldots,N$ due to $|\Omega|=1$.
	
\noindent {\bf Step 2.} (Entropy dissipation due to diffusion)\\
By using Poincar\'e's inequality
\begin{equation}\label{Poincare}
\|\nabla f\|^2 \geq C_P\|f - \overline{f}\|^2, \qquad \text{ for all } f\in H^1(\Omega),
\end{equation}
we have
\begin{align} 
\frac{1}{2}\mathcal{D}(W|X_{\infty}) &\geq \sum_{i=1}^{N}d_i\int_{\Omega}\frac{|\nabla w_i|^2}{u_{i,\infty}}dx	
\geq C_P\sum_{i=1}^{N}d_i\int_{\Omega}\frac{|w_i - \overline{w_i}|^2}{u_{i,\infty}}dx\nonumber \\
&\geq C_P\min\{d_1,d_2,\ldots,d_N\}\,\mathcal{E}(W-\overline{W}|X_{\infty}).
\label{e8}
\end{align}

\noindent {\bf Step 3.} (Entropy dissipation due to reactions) \hfil\\
From \eqref{e7} and \eqref{e8}, it remains to control 
		$$
			\mathcal{E}(\overline{W}|X_{\infty}) =  \sum_{i=1}^{N}\frac{\overline{w_i}^2}{u_{i,\infty}}.
		$$
By using Jensen's inequality we have, recalling that $|\Omega| = 1$,
\begin{align}
\frac{1}{2}\mathcal{D}(W|X_{\infty}) &\geq \frac{1}{2}\sum_{i,j=1;i<j}^{N}(a_{ij}u_{j,\infty} + a_{ji}u_{i,\infty})\int_{\Omega}\left( \frac{w_i}{u_{i,\infty}} - \frac{w_j}{u_{j,\infty}}\right)^2dx\nonumber\\
&\geq \frac{1}{2}\sum_{i,j=1;i<j}^{N}(a_{ij}u_{j,\infty} + a_{ji}u_{i,\infty})\left( \frac{\overline{w_i}}{u_{i,\infty}} - \frac{\overline{w_j}}{u_{j,\infty}}\right)^2dx.\label{e9}
\end{align}
It then remains to prove that
\begin{equation}\label{e10}
\frac{1}{2}\sum_{i,j=1;i<j}^{N}(a_{ij}u_{j,\infty} + a_{ji}u_{i,\infty})\left( \frac{\overline{w_i}}{u_{i,\infty}} - \frac{\overline{w_j}}{u_{j,\infty}}\right)^2dx \geq \gamma \sum_{i=1}^{N}\frac{\overline{w_i}^2}{u_{i,\infty}}
\end{equation}
for some $\gamma>0$. Note that if both reactions $S_i \rightarrow S_j$ and $S_j \rightarrow S_i$ do not appear in the reaction network, then we have $a_{ij} = a_{ji} = 0$ and thus
$$
	a_{ij}u_{j,\infty} + a_{ji}u_{i,\infty} = 0.
$$
Hence, the expression
$$
\sum_{i,j=1;i<j}^{N}(a_{ij}u_{j,\infty} + a_{ji}u_{i,\infty})\left( \frac{\overline{w_i}}{u_{i,\infty}} - \frac{\overline{w_j}}{u_{j,\infty}}\right)^2dx
$$
may not contain all pairs $(i,j)$ with $i\not=j$. However, the weak reversibility of the network allows to make all pairs $(i,j)$ with $i\not=j$ appear 
in the following sense: There exists an explicit constant $\xi>0$ such that
\begin{equation}\label{e11}
\sum_{i,j=1;i<j}^{N}(a_{ij}u_{j,\infty} + a_{ji}u_{i,\infty})\left( \frac{\overline{w_i}}{u_{i,\infty}} - \frac{\overline{w_j}}{u_{j,\infty}}\right)^2dx \geq \xi\sum_{i,j=1;i<j}^{N}\left(\frac{\overline{w_i}}{u_{i,\infty}} - \frac{\overline{w_j}}{u_{j,\infty}}\right)^2.
\end{equation}
Indeed, assume that $a_{ij} = a_{ji} = 0$ for some $i\not=j$. Due to the weak reversibility of the network, there exists a path from $S_i$ to $S_j$ as follows
$$
S_i \equiv S_{j_1} \xrightarrow{a_{j_2j_1}} S_{j_2} \xrightarrow{a_{j_3j_2}} \ldots \xrightarrow{a_{j_{r}j_{r-1}}} S_{j_{r}} \equiv S_j
$$
		with $r\geq 3$ and $a_{j_kj_{k-1}} > 0$ for all $k = 2,3,\ldots, r$. Thus, with 
		\[
			{\color{black}0<\sigma = \min_{(a_{ij}, a_{ji})\not=(0,0);1\le i<j\le N}\{a_{ij}u_{i,\infty} + a_{ji}u_{j,\infty} \} \leq \min_{2\leq k\leq r}\{a_{j_kj_{k-1}}u_{j_{k-1},\infty} + a_{j_{k-1}j_k}u_{j_k,\infty} \}}
		\]
	 we have
\begin{multline}
\sum_{k=2}^{r}(a_{j_kj_{k-1}}u_{j_{k-1},\infty} + a_{j_{k-1}j_k}u_{j_{k},\infty})\left(\frac{\overline{w_{j_{k}}}}{u_{j_k,\infty}} - \frac{\overline{w_{j_{k-1}}}}{u_{j_{k-1},\infty}}\right)^2\\
\geq {\sigma}\sum_{k=2}^{r}\left(\frac{\overline{w_{j_{k}}}}{u_{j_k,\infty}} - \frac{\overline{w_{j_{k-1}}}}{u_{j_{k-1},\infty}}\right)^2\\
\geq \frac{\sigma}{r-1}\left(\frac{\overline{w_{j_{1}}}}{u_{j_1,\infty}} - \frac{\overline{w_{j_{r}}}}{u_{j_{r},\infty}}\right)^2
= \frac{\sigma}{N-1}\left(\frac{\overline{w_{i}}}{u_{i,\infty}} - \frac{\overline{w_{j}}}{u_{j,\infty}}\right)^2.
			\label{e12}
\end{multline}
Since there {\color{black} are less than $N(N-1)/2$} pairs $(i,j)$ with $a_{ij} = a_{ji} = 0$, 
we can repeat this procedure to finally get \eqref{e11} {\color{black}with $\xi = 2\sigma/(N(N-1)^2)$}. From \eqref{e10} and \eqref{e11}, we are left to find a constant 
$\gamma>0$ satisfying
\begin{equation}\label{e13} 
\sum_{i,j=1;i<j}^{N}\left(\frac{\overline{w_i}}{u_{i,\infty}} - \frac{\overline{w_j}}{u_{j,\infty}}\right)^2 \geq \frac{2\gamma}{\xi}\sum_{i=1}^{N}\frac{\overline{w_i}^2}{u_{i,\infty}}
\end{equation}
with the constraint of the conserved zero total mass 
\begin{equation}\label{e14}
	\sum_{i=1}^{N}\overline{w_i} = 0.
\end{equation}
Because of \eqref{e14},
\begin{equation}\label{e14_1}
	\sum_{i=1}^{N}\overline{w_i}^2 = -\sum_{i,j=1;i\neq j}^{N}\overline{w_i}\,\overline{w_j}= -2\sum_{i,j=1;i<j}^{N}\overline{w_i}\,\overline{w_j}.
\end{equation}
Therefore, we can estimate for $C = \min_{1\le i<j\le N}\frac{1}{u_{i,\infty}u_{j,\infty}}$
\begin{multline}
\sum_{i,j=1;i<j}^{N}\left(\frac{\overline{w_i}}{u_{i,\infty}} - \frac{\overline{w_j}}{u_{j,\infty}}\right)^2 
\geq \min_{i<j}\frac{1}{u_{i,\infty}u_{j,\infty}}\sum_{i<j}u_{i,\infty}u_{j,\infty}\left(\frac{\overline{w_i}}{u_{i,\infty}} - \frac{\overline{w_j}}{u_{j,\infty}}\right)^2 \\
\geq -2\min_{i<j}\frac{1}{u_{i,\infty}u_{j,\infty}}\sum_{i<j}\overline{w_i}\;\overline{w_j}
= \min_{i<j}\frac{1}{u_{i,\infty}u_{j,\infty}}\sum_{i=1}^{N}\overline{w_i}^2
\geq \min_{i<j}\frac{1}{u_{i,\infty}u_{j,\infty}}\sum_{i=1}^{N}\frac{\overline{w_i}^2}{u_{i,\infty}}.
\label{e14_2} 
\end{multline}
In conclusion, we have proved \eqref{e13} {\color{black}with $\gamma = \frac{\xi}{2}\min_{i<j}\frac{1}{u_{i,\infty}u_{j,\infty}}$}, which in combination with \eqref{e11} implies \eqref{e10} and thus completes the proof of this Lemma.
\end{proof}

\begin{theorem}[Convergence to Equilibrium]\hfil\label{Convergence}\\
\textcolor{black}{
Consider (w.l.o.g) a strongly connected subnetwork $\mathcal{N}$
of a weakly reversible first order reaction network. 
Assume for $\mathcal{N}$ that 
the diffusion coefficients $d_i$ are positive for all $i = 1,2,\ldots, N$, and the initial mass $M$ is positive. 
}

Then, the unique global solution to \eqref{VectorSystem_ReCall} 
converges to the unique positive equilibrium $X_{\infty}$ in the following sense:
\begin{equation*}
\sum_{i=1}^{N}\int_{\Omega}\frac{|u_i(t) - u_{i,\infty}|^2}{u_{i,\infty}}dx \leq e^{-\lambda t}\sum_{i=1}^{N}\int_{\Omega}\frac{|u_{i,0} - u_{i,\infty}|^2}{u_{i,\infty}}dx,
\end{equation*}
where the constant $\lambda >0$ {\color{black}is computed as in Lemma \ref{EEDEstimate}}.
\end{theorem}
\begin{proof}
	From Lemma \ref{EEDEstimate} we have
	\begin{equation*}
		\frac{d}{dt}\mathcal{E}(X-X_{\infty}|X_{\infty}) = -\mathcal{D}(X-X_{\infty}|X_{\infty}) \leq \lambda\, \mathcal{E}(X - X_{\infty}|X_{\infty}).
	\end{equation*}
	By Gronwall's inequality,
	\begin{equation*}
		\mathcal{E}(X(t) - X_{\infty}|X_{\infty}) \leq e^{-\lambda t}\mathcal{E}(X_0 - X_{\infty}|X_{\infty}),
	\end{equation*}
	and the proof is complete.
\end{proof}
\textcolor{black}{
\begin{proof}[Proof of Theorem \ref{FirstResult}]
Theorem \ref{FirstResult} is a direct consequence of Theorem \ref{Convergence} and the partition of 
weakly reversible first order reaction network into strongly connected components.
\end{proof}
} 
\medskip

We now turn to the case of degenerate diffusion, where some of the diffusion coefficients $d_i$ can be zero. In the proof of the Theorem \ref{Convergence}, we have used non-degenerate diffusion of all species in order to control distance of the concentrations to their spatial averages (see estimate \eqref{e8}). This procedure must thus be adapted in the case of degenerate diffusion. 

It was already proven in \cite{DeFe_Con, BaFeEv14, MiHaMa14} that even if some diffusion coefficients vanish, one can still show exponential convergence to equilibrium provided reversible reactions. 
The technique used in these mentioned references is based on the fact that diffusion of one specie, which is connected through a reversible 
reaction with another specie, induces a indirect kind of "diffusion effect" to the latter specie. 

We will prove that this principle is still valid for weakly reversible reaction networks as considered in this section.
\begin{theorem}[Convergence to Equilibrium with Degenerate Diffusion]\label{ConvDegDiffusion}\hfil\\
\textcolor{black}{
Consider (w.l.o.g) a strongly connected subnetwork $\mathcal{N}$ 
of a weakly reversible first order reaction network.
Assume for $\mathcal{N}$ that the initial mass $M$ is positive. 
Moreover, assume that at least one diffusion coefficient $d_i$ is positive for some $i = 1, 2, \ldots, N$. 
}
	
	Then, the solution to \eqref{VectorSystem_ReCall} converges exponentially to equilibrium via the following estimate
	\begin{equation*}
		\sum_{i=1}^{N}\int_{\Omega}\frac{|u_i(t) - u_{i,\infty}|^2}{u_{i,\infty}}dx \leq e^{-\lambda' t}\sum_{i=1}^{N}\int_{\Omega}\frac{|u_{i,0} - u_{i,\infty}|^2}{u_{i,\infty}}\,dx,
	\end{equation*}
	for some explicit rate $\lambda'>0$ {\color{black}which depends explicitly on $A$, $\Omega$, $D$ and $M$}..
\end{theorem}
\begin{proof}
We aim for a similar entropy-entropy dissipation inequality as stated in Lemma \eqref{EEDEstimate}, i.e. we want to find a constant $\lambda'>0$ such that
\begin{equation}\label{e15}
\mathcal{D}(W|X_{\infty}) \geq \lambda'\, \mathcal{E}(W|X_{\infty}) = \lambda'\,[\mathcal E(W-\overline{W}|X_{\infty}) + \mathcal E(\overline{W}|X_{\infty})].
\end{equation}
Due to the degenerate diffusion, the diffusion part of $\mathcal{D}(W|X_{\infty})$ is insufficent to control $\mathcal E(W - \overline{W}|X_{\infty})$ as in \eqref{e8}, since some of diffusion coefficients can be zero. This difficulty can be resolved by quantifying the fact that diffusion of one specie is transferred to another species when connected via a weakly reversible reaction path.
Without loss of generality, we assume that $d_1>0$ and estimate $\mathcal{D}(W|X_{\infty})$ by
\begin{equation}\label{e16}
		\mathcal{D}(W|X_{\infty}) \geq d_1\int_{\Omega}\frac{|\nabla w_1|^2}{u_{1,\infty}}\,dx + \sum_{i,j=1;i<j}^{N}(a_{ij}u_{j,\infty}+a_{ji}u_{i,\infty})\int_{\Omega}\left(\frac{w_i}{u_{i,\infty}} - \frac{w_j}{u_{j,\infty}}\right)^2dx.
\end{equation}
By arguments similar to \eqref{e11} and \eqref{e12}, we have
\begin{equation}\label{e17}
		\mathcal{D}(W|X_{\infty}) \geq d_1\int_{\Omega}\frac{|\nabla w_1|^2}{u_{1,\infty}}dx + \xi\sum_{i,j=1;i<j}^{N}\int_{\Omega}\left(\frac{w_i}{u_{i,\infty}} - \frac{w_j}{u_{j,\infty}}\right)^2dx.
\end{equation}
To control $\mathcal{E}(W - \overline{W}|X_{\infty})$, we use the following estimate for all $i=2,3,\ldots,N$:
\begin{equation}\label{e18}
\int_{\Omega}\frac{|\nabla w_1|^2}{u_{1,\infty}}\,dx + \int_{\Omega}\left(\frac{w_1}{u_{1,\infty}} - \frac{w_i}{u_{i,\infty}}\right)^2dx \geq \beta\int_{\Omega}\frac{|w_i - \overline{w_i}|^2}{u_{i,\infty}}\,dx, 
\end{equation}
{\color{black}with $\beta = \frac{1}{2u_{1,\infty}}\min\left\{\frac{C_P}{u_{1,\infty}}, 1 \right\}$}: Indeed, thanks to Poincar\'{e}'s inequality $\|\nabla f\|^2 \geq C_P\|f - \overline{f}\|^2$, we estimate for various sufficiently small constants $C$
\begin{align}
\int_{\Omega}\frac{|\nabla w_1|^2}{u_{1,\infty}}\,dx &+ \int_{\Omega}\left(\frac{w_1}{u_{1,\infty}} - \frac{w_i}{u_{i,\infty}}\right)^2dx\nonumber 
\geq \int_{\Omega}\left[C_P\,\frac{|w_1 - \overline{w_1}|^2}{u_{1,\infty}} + \left(\frac{w_1 - \overline{w_1}}{u_{1,\infty}} + \frac{\overline{w_1}}{u_{1,\infty}}- \frac{w_i}{u_{i,\infty}}\right)^2\right]dx\nonumber \\
&\geq \frac 12\min\left\{\frac{C_P}{u_{1,\infty}}, 1 \right\}\int_{\Omega}\left(\frac{\overline{w_1}}{u_{1,\infty}}- \frac{w_i}{u_{i,\infty}}\right)^2dx\nonumber \\
&= \frac 12\min\left\{\frac{C_P}{u_{1,\infty}}, 1 \right\}\int_{\Omega}\left(\frac{\overline{w_1}}{u_{1,\infty}}- \frac{\overline{w_i}}{u_{i,\infty}} + \frac{\overline{w_i}}{u_{i,\infty}} - \frac{w_i}{u_{i,\infty}}\right)^2dx\nonumber \\
&= \frac 12\min\left\{\frac{C_P}{u_{1,\infty}}, 1 \right\}\int_{\Omega}\left(\frac{\overline{w_1}}{u_{1,\infty}}- \frac{\overline{w_i}}{u_{i,\infty}}\right)^2dx + \frac 12\min\left\{\frac{C_P}{u_{1,\infty}}, 1 \right\}\int_{\Omega}\left(\frac{\overline{w_i}}{u_{i,\infty}} - \frac{w_i}{u_{i,\infty}}\right)^2dx\nonumber \\
&\geq \frac{1}{2u_{1,\infty}}\min\left\{\frac{C_P}{u_{1,\infty}}, 1 \right\}\int_{\Omega}\frac{|w_i - \overline{w_i}|^2}{u_{i,\infty}}\,dx.
	\label{e19} 
\end{align}
	
Now, thanks to \eqref{e17} and \eqref{e18}
\begin{align}
\mathcal{D}(W|X_{\infty}) &\geq {\color{black}\min\left\{\frac{d_1}{N},\frac{\xi}{2} \right\}\beta}\sum_{i=1}^{N}\int_{\Omega}\frac{|w_i - \overline{w_i}|^2}{u_{i,\infty}}dx + \frac{\xi}{2}\sum_{i,j=1;i<j}^{N}\int_{\Omega}\left(\frac{w_i}{u_{i,\infty}} - \frac{w_j}{u_{j,\infty}}\right)^2dx\nonumber \\
&\geq {\color{black}\min\left\{\frac{d_1}{N},\frac{\xi}{2} \right\}\beta}\,\mathcal{E}(W - \overline{W}|X_{\infty}) + \frac{\xi}{2}\sum_{i,j=1;i<j}^{N}\int_{\Omega}\left(\frac{\overline{w_i}}{u_{i,\infty}} - \frac{\overline{w_j}}{u_{j,\infty}}\right)^2dx\nonumber \\
&\geq {\color{black}\min\left\{\frac{d_1}{N},\frac{\xi}{2} \right\}\beta}\mathcal{E}(W - \overline{W}|X_{\infty}) + \frac{\gamma}{4}\mathcal{E}(\overline{W}|X_{\infty}) \qquad\qquad (\text{by using }\eqref{e13})\nonumber \\
&\geq \lambda'\,\mathcal{E}(W|X_{\infty})\label{e20}
\end{align}
{\color{black}with $\lambda' = \min\left\{\frac{\beta d_1}{N}, \frac{\xi \beta}{2}, \frac{\gamma}{4}\right\}$}.
Thus \eqref{e15} is proved and the proof is complete.
\end{proof}

\textcolor{black}{
\begin{proof}[Proof of Theorem \ref{DegenerateDiffusion}]
Theorem \ref{DegenerateDiffusion} is a direct consequence of Theorem \ref{ConvDegDiffusion} and the partition of 
weakly reversible first order reaction network into strongly connected components.
\end{proof}
} 
\medskip

\begin{remark}\label{DiffReactCoupl}
The estimate \eqref{e18} is usually interpreted as follows: the sum of the  dissipation due to the diffusion of $w_1$ and the dissipation caused by the reaction between $w_1$ and $w_i$ are bounded below by \eqref{e19}, which is essentially a diffusion dissipation term of the specie $w_i$ (after having applied Poincar\'e's inequality). In this sense, a "diffusion effect" has been transferred onto $w_i$.

We remark that while the presented proof for the linear case is straightforward, the proof of an analogous estimate to \eqref{e18} in  nonlinear cases turns out to be quite tricky. Readers are referred to \cite{DeFe_Con} or \cite[Lemma 3.6]{BaFeEv14} for more details.
\end{remark}

\section{Non-weakly reversible networks}\label{non-weakly}

In this section, we consider \textcolor{black}{(w.l.o.g.) reaction networks $\mathcal{N}$ which are not weakly reversible, yet form one linkage class. Thus, the corresponding directed graph $G$
is connected yet not strongly connected  (i.e. the underlying undirected graph of $G$ is connected)}. We will show that in the large time behaviour, each specie tends exponentially fast either to zero or to a positive equilibrium value depending on its position in the graph representing the network. 

For weakly reversible reaction-diffusion networks (corresponding to strongly connected graphs), it was proven in Section \ref{weakly} that each specie converges exponentially fast to a unique, positive equilibrium value, which is given explicitly in terms of the reaction rates and the conserved initial total mass. 

For non weakly reversible reaction networks, however, we will show that while the equilibria are still unique and attained exponentially fast, the equilibrium values are in general no longer explicitly given but depend on the position in the graph in general and on the history of the concentrations of the influencing species in particular. 

Moreover, since non weakly reversible reaction networks \eqref{VectorSystem_ReCall} may no longer have positive equilibria, the relative entropy method used in Section \ref{weakly} is not directly applicable. Nevertheless, we will see that the relative entropy and the ideas of the entropy method still play the essential role our analysis of non-weakly reversible networks.

As the large time behaviour of the species depend on their position within the network, we need to first state some important properties of the graph $G$. The following Lemmas \ref{newgraph} and \ref{TopoOrder} are well known in graph theory. We refer the reader to the book \cite{BJG08} for a reference.
\begin{lemma}[Strongly connected components form acyclic graphs $G^C$] \label{newgraph}\hfil\\
Let $G$ be a directed graph which is connected, that is the underlying undirected graph of $G$ is connected, but not strongly connected such that the graph $G$ contains at least 
$r\geq 2$ strongly connect components, which we shall denote by $C_1, C_2, \ldots, C_r$. Thus, we can define a directed graph $G^C$ of strongly connected components as follows
\begin{itemize}
\item[-] $G^C$ has as nodes the $r$ strongly connected components $C_1, C_2, \ldots, C_r$,
\item[-] for two nodes $C_i$ and $C_j$ of $G^C$, if there exists a reaction $C_i \ni S_k \xrightarrow{a_{\ell k}} S_\ell\in C_j$ with $a_{\ell k} >0$, then we define a directed edge $C_i \rightarrow C_j$ of $G^C$.
\end{itemize}
Then, the directed graph $G^C$ is acyclic, that is $G^C$ does not contain any cycles.
\end{lemma}
\begin{proof}
The proof can be found in e.g. \cite[Chapter 1]{BJG08} and shows that if $G^C$ would contain a cycle then this cycle should have been contained in a strongly connected component in the first place. 
\end{proof}

\begin{lemma}[{Topological order of acyclic graphs, \cite[Chapter 1]{BJG08}}]\label{TopoOrder}\hfil\\
There exists a reordering of the nodes of $G^C$ in such a way that for all direct edges $C_i \rightarrow C_j$ we always have $i < j$.
\end{lemma}

From now on, we will always consider topologically ordered graphs $G^C$. For each $i=1,2,\ldots, N$, we denote by $N_i$ the number of species belonging to $C_i$. For notational convenience later on, we shall set $L[0] = 0$ and
introduce the cumulative number $L[i]$ of the species contained in all strongly connected components up to $C_i$, i.e. 
\begin{equation}\label{L} 
	L[i] = N_1 + N_2 + \ldots + N_i \qquad \text{ for all } i=1,2,\ldots, r.
\end{equation}
We then reorder the species of the network $\mathcal{N}$ in such a order that the species belong to the component $C_i$ are $S_{L[i-1]+1}, S_{L[i-1]+2}, \ldots, S_{L[i]}$ for all $i=1,2,\ldots, N$. 
\medskip

Each component $C_i$ belongs to one of the following three types:
\begin{itemize}
	\item {\it Source component}: $C_i$ is a source component if there is no in-flow to $C_i$, i.e. there does not exist an edge $C_i\not\ni S_k\rightarrow S_j \in C_i$,
	\item {\it Target component}: $C_i$ is a target component if there is no out-flow from $C_i$, i.e. there does not exist an edge $C_i\ni S_k \rightarrow S_j\not\in C_i$,
	\item {\it Transmission component}: If $C_i$ is neither a source component nor a target component, then $C_i$ is called a transmission component.
\end{itemize}
\medskip

The above classification of strongly connected components greatly improves the notation of the corresponding dynamics, which quantifies the behaviour of the species belonging to the three types of components.
In the following, we denote by $X_{i} = (u_{L[i-1]+1}, u_{L[i-1]+2}, \ldots, u_{L[i]})^T$ the concentration vector of the species belonging to $C_i$. 

The evolution of the species belonging to a component $C_i$ depends 
on the type of $C_i$:
\begin{itemize}
\item[(i)] For a source component $C_i$, the system for $X_i$ is of the form
\begin{equation}\label{SourceSystem} 
\begin{cases}
	\partial_tX_{i} - D_i\Delta X_{i} = A_iX_{i} - F_{i}^{out}X_{i}, &x\in\Omega, \qquad t>0,\\
	\partial_{\nu}X_{i} = 0, &x\in\partial\Omega, \qquad t>0,\\
	X_{i}(x,0) = X_{i,0}(x), &x\in\Omega,
\end{cases}
\end{equation} 
where the diffusion matrix $D_i$ is 
\begin{equation}\label{DiffMatrix} 
	D_i = \text{diag}(d_{L[i-1]+1}, d_{L[i-1]+2}, \ldots, d_{L[i]}) \in \mathbb{R}^{N_i\times N_i},
\end{equation}
the reaction matrix $A_i$ is
\begin{equation}\label{SelfReacMatrix} 
	A_i = (a_{L[i-1]+k,L[i-1]+\ell})_{1\leq k, \ell \leq N_i} \in \mathbb{R}^{N_i\times N_i},
\end{equation}
and the out flow matrix is defined as
\begin{equation}\label{OutFlow} 
	F_{i}^{out} = \text{diag}(f_{L[i-1]+1}, f_{L[i-1]+2}, \ldots, f_{L[i]}) \in \mathbb{R}^{N_i\times N_i}
\end{equation}
with 
$$
	f_{L[i-1]+k} = \sum_{\ell=L[i]+1}^{N}a_{\ell,L[i-1]+k} \qquad \forall k=1,2,\ldots, N_i,
$$
where the lower summation index $L[i]+1$ follows for the topological order of the graph $G^C$. 

Roughly speaking, $f_{L[i-1]+k}$ is the sum of all the reaction rates from the specie $S_{L[i-1]+k}$ to species outside of $C_i$. It may happen that $f_{L[i-1]+k} = 0$ for some $k=1,2,\ldots,N$, but there exists at least one $k_0$ such that $f_{L[i-1]+k_0} > 0$ since $C_i$ is a source component.\\
	
\item[(ii)] If $C_i$ is a transmission component, the system for $X_i$ writes as
\begin{equation}\label{TransmissionSystem} 
	\begin{cases}
		\partial_tX_{i} - D_i\Delta X_{i} = \mathcal{F}_{i}^{in} + A_iX_{i} - F_{i}^{out}X_i, &x\in\Omega,\quad t>0,\\
		\partial_{\nu}X_{i} = 0, &x\in\partial\Omega, \quad t>0,\\
		X_{i}(x,0) = X_{i,0}(x), &x\in\Omega,
	\end{cases}
\end{equation} 
where the diffusion matrix $D_i$, the reaction matrix $A_i$ and the out flow matrix$F_{i}^{out}$ are defined as above in \eqref{DiffMatrix}, \eqref{SelfReacMatrix} and \eqref{OutFlow}, respectively. The in-flow vector  $\mathcal{F}_{i}^{in}$ is defined by 
\begin{equation}\label{InFlow}
	\mathcal{F}_{i}^{in} = \begin{pmatrix}
		z_{L[i-1]+1}\\
		z_{L[i-1]+2}\\
			\ldots \\
		z_{L[i]}\\
		\end{pmatrix}
\quad \text{ with } \quad z_{L[i-1]+\ell} = \sum_{k=1}^{L[i-1]}a_{L[i-1]+\ell, k}u_k. 
\end{equation}
We remark that by studying all components $C_i$ within the 
topological order of $G^C$, the dynamics of the previous components $C_1, C_2, \ldots, C_{i-1}$ is already known at the time we analyse the component $C_i$. Thus, in system \eqref{TransmissionSystem} the in-flow vector $\mathcal{F}_{i}^{in}$ can be considered as a given external in-flow.\\
	
\item[(iii)] If $C_i$ is a target component, we can write
\begin{equation}\label{TargetSystem} 
	\begin{cases}
	\partial_tX_{i} - D_i\Delta X_{i} = \mathcal{F}_{i}^{in} + A_iX_{i}, &x\in\Omega,\quad t>0,\\
	\partial_{\nu}X_{i} = 0, &x\in\partial\Omega, \quad t>0,\\
	X_{i}(x,0) = X_{i,0}(x), &x\in\Omega,
\end{cases}
\end{equation} 
where the reaction matrix $A_i$ and the in-flow $\mathcal{F}_{i}^{in}$ are defined in the same way as above in \eqref{SelfReacMatrix} and \eqref{InFlow}.
\end{itemize}
\medskip

By modifying the relative entropy method in Section \ref{weakly}, we obtain the
%
\begin{proof}[Proof of Theorem \ref{SourceTransmission}]
Since the ongoing outflow vanishes the mass of all source components and subsequently all transmission components, the corresponding equilibrium values are expected to be zero and the relative entropy method used for weakly reversible networks is not directly applicable here. We instead introduce a concept of "artificial equilibrium states with normalised mass" for these components, which allows to derive a quadratic entropy-like functional, which can be proved to decay exponentially. Due to their different dynamics, we have to distinguish the two cases: $C_i$ is a source component and $C_i$ is a transmission component.

The aim of the proof is to show that if $C_i$ is a source or a transmission component then for all $k = 1, \ldots, N_i$,
\begin{equation}\label{source_transmission}
	\|u_{L[i-1]+k}(t)\|_{L^2(\Omega)}^2 \leq K_ie^{-\lambda_i t}, \qquad \text{ for all } t\geq 0,
\end{equation}
for explicit constants $K_i>0$ and $\lambda_i>0$.
\medskip

In order to simplify the notation, we shall denote
\begin{equation}\label{newnotation}
v_k = u_{L[i-1]+k},\quad \text{and}\quad b_{k,\ell} = a_{L[i-1]+k, L[i-1]+\ell},\qquad \text{for all } 1\leq k, \ell \leq N_i. 
\end{equation}
Then, the concentration vector $X_i$ and the reaction matrix $A_i$ can be rewritten as 
\begin{equation*}
		X_i = (v_1, v_2, \ldots, v_{N_i})\qquad \text{ and } \qquad A_i = (b_{k,\ell})_{1\leq k, \ell \leq N_i}.
\end{equation*}
Note that the index $i$ for the component $C_i$ is fixed.\\
	
\noindent {\bf Case 1:} $C_i$ is a source component.
	
We recall the corresponding system from \eqref{SourceSystem}
\begin{equation}\label{SourceSystem_Recall}
	\begin{cases}
		\partial_tX_{i} - D_i\Delta X_{i} = A_iX_{i} - F_{i}^{out}X_{i}, &x\in\Omega, \qquad t>0,\\
		\partial_{\nu}X_{i} = 0, &x\in\partial\Omega, \qquad t>0,\\
		X_{i}(x,0) = X_{i,0}(x), &x\in\Omega.
	\end{cases}
\end{equation} 
We now introduce an artificial equilibrium state $X_{i,\infty} = (v_{1,\infty}, v_{2,\infty}\ldots, v_{N_i, \infty})^T$ with normalised mass to \eqref{SourceSystem_Recall}, which is defined as the solution of the system 
\begin{equation}\label{SourceEqui}
	\begin{cases}
		A_iX_{i,\infty} = 0,\\
		v_{1,\infty} + v_{2,\infty} + \ldots + v_{N_i,\infty} = 1.
	\end{cases}
\end{equation}
It follows from Lemma \ref{UniqueEqui} that there exists a unique positive solution $X_{i,\infty}$ to \eqref{SourceEqui}. 
Here we notice that $X_{i,\infty}$ balances all reactions within $C_i$ while the total mass contained in $X_{i,\infty}$ is normalised to one. 
\medskip

In the following we will study the evolution of the quadratic entropy-like functional
\begin{equation}\label{ReEn_Source}
	\mathcal{E}(X_{i}|X_{i,\infty}) = \sum_{k=1}^{N_i}\int_{\Omega}\frac{|v_{k}|^2}{v_{k,\infty}}dx.
\end{equation}
By similar calculations as in Lemma \ref{ExplicitEnDiss}, we obtain the time derivative of this quadratic functional
\begin{align}\label{EnDiss_Source} 
	\mathcal{D}(X_{i}|X_{i,\infty}) &= -\frac{d}{dt}\mathcal{E}(X_{i}|X_{i,\infty})\nonumber \\
	&= 2\sum_{k=1}^{N_i}d_{L[i-1]+k}\int_{\Omega}\frac{|\nabla v_{k}|^2}{v_{k,\infty}}dx\nonumber\\
	&\quad + \sum_{k,\ell=1;k<\ell}^{N_i}(b_{k,\ell}v_{\ell,\infty} + b_{\ell,k}v_{k,\infty})\int_{\Omega}\left(\frac{v_{k}}{v_{k, \infty}} - \frac{v_{\ell}}{v_{\ell, \infty}}\right)^2dx\nonumber\\
	&\quad + 2\sum_{k=1}^{N_i}f_{L[i-1]+k}\int_{\Omega}\frac{|v_k|^2}{v_{k,\infty}}dx.
\end{align}
We remark that since $C_i$ is a source component, there exists an index $k_0 \in \{1,2,\ldots,N_i\}$ 
such that the out-flow $f_{L[i-1]+k_0}>0$ is positive. 
Then, an estimate similar to \eqref{e11} gives for various constants $C$
\begin{align}
\mathcal{D}(X_{i}|X_{i,\infty})&\geq \xi\sum_{k,\ell=1;k<\ell}^{N_i}\int_{\Omega}\Bigl(\frac{v_{k}}{v_{k, \infty}} - \frac{v_{\ell}}{v_{\ell, \infty}}\Bigr)^{\!2}dx + 2f_{L[i-1]+k_0}\int_{\Omega}\frac{|v_{k_0}|^2}{v_{k_0,\infty}}\,dx\nonumber \\
&\geq {\color{black}\min\{\xi/2, f_{L[i-1]+k_0}/2N_i\}}\sum_{\ell=1;\ell\not=k_0}^{N_i}\int_{\Omega}\biggl[\Bigl(\frac{v_{\ell}}{v_{\ell, \infty}} - \frac{v_{k_0}}{v_{k_0, \infty}}\Bigr)^{\!2} + \frac{|v_{k_0}|^2}{v_{k_0,\infty}}\biggr]dx\nonumber\\
&\quad\, + f_{L[i-1]+k_0}\int_{\Omega}\frac{|v_{k_0}|^2}{v_{k_0,\infty}}dx\nonumber\\
&\geq \lambda_i\sum_{\ell=1}^{N_i}\int_{\Omega}\frac{|v_{\ell}|^2}{v_{\ell,\infty}}\,dx
= \lambda_i\, \mathcal{E}(X_{i}|X_{i,\infty})\label{EnEnDissEst_Source} 
\end{align}
{\color{black}with $\lambda_i = {\color{black}\min\{\xi/4, f_{L[i-1]+k_0}/4N_i\}}$}. It follows that
\begin{equation*}
	\frac{d}{dt}\mathcal{E}(X_{i}|X_{i,\infty}) = -\mathcal{D}(X_{i}|X_{i,\infty})\leq -\lambda_i\, \mathcal{E}(X_{i}|X_{i,\infty}),
\end{equation*}
and thus
\begin{equation*}
	\sum_{k=1}^{N_i}\int_{\Omega}\frac{|v_k(t)|^2}{v_{k,\infty}}dx = \mathcal{E}(X_{i}(t)|X_{i,\infty}) \leq e^{-\lambda_i t}\mathcal{E}(X_{i,0}|X_{i,\infty}),
\end{equation*}
or equivalently
\begin{equation*}
	\|u_{L[i-1]+k}(t)\|^2 \leq e^{-\lambda_i t}\mathcal{E}(X_{i,0}|X_{i,\infty})\max_{1\le i\le N_i}\{v_{i,\infty}\} \qquad \text{ for all } t>0, \; \text{ for all } k=1,2,\ldots,N_i,
\end{equation*}
{\color{black}which proves \eqref{source_transmission} with $K_i = \mathcal{E}(X_{i,0}|X_{i,\infty})\max_{1\le i\le N_i}\{v_{i,\infty}\}$} in the case $C_i$ is a source component.
\medskip		
	
\noindent {\bf Case 2:} $C_i$ is a transmission component.

By recalling that the components $C_i$ are topologically ordered, we can assume without loss of generality  
that $u_\ell$, with $\ell=1,2,\ldots,L[i-1]$, obeys the following exponential decay
\begin{equation}\label{Decay}
	\|u_\ell(t)\|^2 \leq K^*e^{-\lambda^* t}, \qquad \ell=1,2,\ldots,L[i-1], \quad\text{ for all } t>0.
\end{equation}
for  {\color{black}$0<\lambda^* = \min\limits_{1\leq k\leq i-1}\lambda_k$ and $K^* = \max\limits_{1\leq k\leq i-1}K_i$}. We also recall the system for $C_i$,
\begin{equation}\label{TransmissionSystem_ReCall}
	\begin{cases}
		\partial_tX_{i} - D_i\Delta X_{i} = \mathcal{F}_{i}^{in} + A_iX_{i} - F_{i}^{out}X_{i}, &x\in\Omega,\quad t>0,\\
		\partial_{\nu}X_{i} = 0, &x\in\partial\Omega, \quad t>0,\\
		X_{i}(x,0) = X_{i,0}(x), &x\in\Omega,
	\end{cases}
\end{equation} 
with $\mathcal{F}_{i}^{in}$ is defined as \eqref{InFlow}. Denote by $X_{i,\infty} = (v_{1,\infty}, \ldots, v_{N_i,\infty})^T$ the artificial equilibrium state of \eqref{TransmissionSystem_ReCall}, which is the unique positive solution to
\begin{equation}\label{TransmissionEqui} 
	\begin{cases}
		A_iX_{i,\infty} = 0,\\
		v_{1,\infty} + v_{2,\infty} + \ldots + v_{N_i,\infty} = 1.
	\end{cases}
\end{equation}
Again, we can compute the time derivative of 
\begin{equation}\label{ReEn_Transmission} 
	\mathcal{E}(X_{i}|X_{i,\infty}) = \sum_{k=1}^{N_i}\int_{\Omega}\frac{|v_k|^2}{v_{k,\infty}}dx
\end{equation} 
as
\begin{align}\label{EnDiss_Transmission} 
\mathcal{D}(X_{i}|X_{i,\infty}) &= -\frac{d}{dt}\mathcal{E}(X_{i}, X_{i,\infty})\nonumber\\
&= 2\sum_{i=1}^{N_i}d_{L[i-1]+k}\int_{\Omega}\frac{|\nabla v_k|^2}{v_{k,\infty}}dx + \sum_{k,\ell=1;k<\ell}^{N_i}(b_{k,\ell}v_{\ell,\infty} + b_{\ell,k}v_{k,\infty})\int_{\Omega}\left(\frac{v_{k}}{v_{k, \infty}} - \frac{v_{\ell}}{v_{\ell, \infty}}\right)^2dx\nonumber\\
	&\quad + 2\sum_{k=1}^{N_i}f_{L[i-1]+k}\int_{\Omega}\frac{|v_k|^2}{v_{k,\infty}}dx-2\sum_{k=1}^{N_i}\int_{\Omega}\biggl(\frac{v_k}{v_{k,\infty}}\sum_{\ell=1}^{L[i-1]}a_{L[i-1]+k,\ell}\,u_{\ell}\biggr)dx.
\end{align}
Because $C_i$ is a transmission component, there exists an index $k_0 \in \{1,\ldots, N_i\}$ such that $f_{L[i-1]+k_0}>0$ is positive. In comparison to \eqref{EnDiss_Source}, the dissipation $\mathcal{D}(X_i|X_{i,\infty})$ in \eqref{EnDiss_Transmission}  has the additional term
\begin{equation*}
	- 2\sum_{k=1}^{N_i}\int_{\Omega}\biggl(\frac{v_k}{v_{k,\infty}}\sum_{\ell=1}^{L[i-1]}a_{L[i-1]+k,\ell}\,u_{\ell}\biggr)dx
\end{equation*} 
to be estimated. Thanks to the decay \eqref{Decay} of $u_{\ell}$, we can estimate
\begin{align}\label{InFlow_Estimate} 
\left|2\sum_{k=1}^{N_i}\int_{\Omega}\biggl(\frac{v_k}{v_{k,\infty}}\sum_{\ell=1}^{L[i-1]}a_{L[i-1]+k,\ell}\,u_{\ell}\biggr)dx\right|
&\leq 2\sum_{k=1}^{N_i}\sum_{\ell=1}^{L[i-1]}a_{L[i-1]+k,\ell}\int_{\Omega}\left|\frac{v_{k}}{v_{k,\infty}}u_{\ell}\right|dx \nonumber\\
&\leq f_{L[i-1]+k_0}\sum_{k=1}^{N_i}\int_{\Omega}\frac{|v_k|^2}{v_{k,\infty}}dx + \kappa\sum_{\ell=1}^{L[i-1]}\|u_{\ell}\|^2\nonumber\\
&\leq f_{L[i-1]+k_0}\sum_{k=1}^{N_i}\int_{\Omega}\frac{|v_k|^2}{v_{k,\infty}}dx + \kappa K^*e^{-\lambda^* t}
\end{align}
{\color{black}with $\kappa = N_iL[i-1]\max\limits_{i<j}\{a_{ij}^2\}/(f_{L[i-1]+k_0}\min\limits_{k}\{v_{k,\infty}\})$}.
Then, with the help of \eqref{InFlow_Estimate}, we estimate
\begin{equation*}
	\begin{aligned}
		\mathcal{D}(X_{i}|X_{i,\infty}) \ge
		\sum_{k,\ell=1;k<\ell}^{N_i}(b_{k,\ell}v_{\ell,\infty} + b_{\ell,k}v_{k,\infty})\int_{\Omega}\left(\frac{v_{k}}{v_{k, \infty}} - \frac{v_{\ell}}{v_{\ell, \infty}}\right)^2dx
			+ f_{L[i-1]+k_0}\int_{\Omega}\frac{|v_{k_0}|^2}{v_{k_0,\infty}}dx - \kappa K^* e^{-\lambda^* t},
	\end{aligned}
\end{equation*} 
and similarly to \eqref{EnEnDissEst_Source}, we obtain for ${\color{black}\overline{\lambda} = \min\{\xi/4, f_{L[i-1]+k_0}/4N_i\}},$
\begin{equation}\label{EnEnDissEst_Transmission}
		\mathcal{D}(X_{i}|X_{i,\infty}) \geq \overline{\lambda}\,\mathcal{E}(X_i|X_{i,\infty}) - \kappa K^* e^{-\lambda^* t}.
\end{equation}
From \eqref{EnEnDissEst_Transmission}, we can use the classic Gronwall lemma to have
\begin{equation*}
			\mathcal{E}(X_{i}(t)|X_{i,\infty}) \leq K_ie^{-\lambda_it},
\end{equation*}
{\color{black}with $\lambda_i = \min\{\overline{\lambda}, \lambda^*\}$ and $K_i = 2\max\{\mathcal{E}(X_{i,0}|X_{i,\infty}), \kappa K^* \}$}, which ends the proof in the case that $C_i$ is a transmission component.
\end{proof}

For a target component, we need to define its corresponding equilibrium state. This equilibrium state balances the 
reactions within the component and has as total mass the sum of the initial total mass of the target component 
plus the total "injected mass" from the other components. In general, 
the injected mass will not be given explicitly but depend on the 
time evolution of the influences species prior to $C_i$ in terms of the topological order. 

\begin{lemma}[Equilibrium state of target components]\label{FinalStateTarget}\hfil\\
For each target component $C_i$, if{
\begin{equation}\label{positivemass}
	\sum\limits_{k=1}^{N_i}\overline{u}_{L[i-1]+k,0} + \sum\limits_{k=1}^{N_i}\sum\limits_{\ell=1}^{L[i-1]}a_{L[i-1]+k,\ell}\int\limits_{0}^{+\infty}\overline{u_{\ell}}(s)ds > 0
\end{equation}}
holds, then there exists a unique positive equilibrium state $X_{i,\infty} = (v_{1,\infty}, v_{2,\infty}, \ldots, v_{N_i,\infty})$ satisfying
\begin{equation}\label{EquiTarget}
	\begin{cases}
		A_iX_{i,\infty} = 0,\\
		\sum\limits_{k=1}^{N_i}v_{k,\infty} = \sum\limits_{k=1}^{N_i}\overline{u}_{L[i-1]+k,0} + \sum\limits_{k=1}^{N_i}\sum\limits_{\ell=1}^{L[i-1]}a_{L[i-1]+k,\ell}\int\limits_{0}^{+\infty}\overline{u_{\ell}}(s)ds.
	\end{cases}
\end{equation}
{Otherwise, if the sum \eqref{positivemass} should be zero, then the initial and the total injected mass 
into the target component $C_i$ is zero and the concentrations of the target component $C_i$ remain zero of all times.}
\end{lemma}
\begin{proof}
By \eqref{Decay} we have for all $\ell = 1, 2, \ldots, L[i-1]$ that $\|u_{\ell}(t)\|^2 \leq K^* e^{-\lambda^*t}$. 
Thus, Jensen's inequality yields
\begin{equation}\label{Jensen}
	\int_{0}^{+\infty}\overline{u_{\ell}}(s)ds \leq \int_{0}^{+\infty}\|u(s)\|^{1/2}_{L^2(\Omega)}ds \leq K^*\int_{0}^{+\infty}e^{-\frac{\lambda^*}{2}s}ds = \frac{2K^*}{\lambda^*}
\end{equation}
and the right hand side of the second equation in \eqref{EquiTarget} is finite. 
Therefore, the existence of a unique $X_{i,\infty}$ satisfying \eqref{EquiTarget} follows from Lemma \ref{UniqueEqui}.
\end{proof}
{
\begin{remark}
{The positive sign in assumption \eqref{positivemass} ensures that either initially or during the ongoing reactions 
positive mass is present/injected into the component $C_i$.} 
When this assumption does not hold, 
then the target component does not possess a positive equilibrium and all of its concentrations remain zero for all times. 
For example, consider the network
	\begin{center}\scalebox{1}[1]
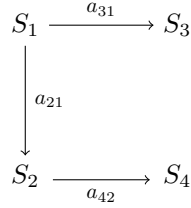
{
		\begin{tikzpicture}
			\node (c) at (2,2) {$S_1$}  node (d) at (2,0) {$S_2$}  node (e) at (4,0) {$S_4$}   node(f) at(4,2){$S_3$};
			\draw[arrows=->] ([xshift =0.5mm,yshift=-0.5mm]d.east) -- node [below] {\scalebox{.8}[.8]{$a_{42}$}} ([xshift=-0.5mm,yshift=-0.5mm]e.west);
			\draw[arrows=->] ([xshift =0.5mm]c) -- node [above] {\scalebox{.8}[.8]{$a_{31}$}} ([xshift=-0.5mm]f);
			\draw[arrows=->] ([yshift =-0.5mm]c) -- node [right] {\scalebox{.8}[.8]{$a_{21}$}} ([yshift=0.5mm]d);
		\end{tikzpicture}} 
	\end{center}	
when the initial data of all species are zero except $S_3$. In this case, the target component $\{S_4\}$ will not ever receive any mass, and thus remains zero for all $t>0$.
\end{remark}}

We now begin the 
\begin{proof}[Proof of Theorem \ref{Target}]

With the notations introduced in \eqref{L} and Lemma \ref{FinalStateTarget}, we identify the indexes in the statement of Theorem \ref{Target} as $i_k = L[i-1]+k$ and the equilibrium state $u_{i_k,\infty} = v_{k,\infty}$ for $k= 1,\ldots, N_i$. The aim now is to prove for all $k=1,\ldots, N_i$,
\begin{equation*}
	\|v_k(t) - v_{k,\infty}\|_{L^2(\Omega)}^2 \leq K_ie^{-\lambda_i t} \quad \text{ for all } t\geq 0
\end{equation*}
for some explicit constants $K_i>0$ and $\lambda_i>0$.

\medskip
We recall the system for a target component $C_i$,
\begin{equation}\label{TargetSystem_Recall}
\begin{cases}
		\partial_tX_{i} - D_i\Delta X_{i} = \mathcal{F}_{i}^{in} + A_iX_{i}, &x\in\Omega,\quad t>0,\\
		\partial_{\nu}X_{i} = 0, &x\in\partial\Omega, \quad t>0,\\
		X_{i}(x,0) = X_{i,0}(x), &x\in\Omega,
\end{cases}
\end{equation}
where
\begin{equation*} 
			\mathcal{F}_{i}^{in} = \begin{pmatrix}
					z_{L[i-1]+1}\\
					z_{L[i-1]+2}\\
					\ldots \\
					z_{L[i]}\\
			\end{pmatrix}
	\quad \text{ with } \quad z_{L[i-1]+\ell} = \sum_{k=1}^{L[i-1]}a_{L[i-1]+\ell, k}\,u_k. 
\end{equation*}
	
Note that the total mass of $C_i$ is not conserved but increases in time due to the in-flow vector $\mathcal{F}^{in}_i$. 
To compute the total mass of $C_i$ at a time $t>0$, we sum up all the equations of 
\eqref{TargetSystem_Recall} then integrating over $\Omega$,
\begin{equation*}
		\frac{d}{dt}\sum_{k=1}^{N_i}\overline{u}_{L[i-1]+k}(t) = \sum_{k=1}^{N_i}\overline{z}_{L[i-1]+k}(t) = \sum_{k=1}^{N_i}\sum_{\ell=1}^{L[i-1]}a_{L[i-1]+k,\ell}\,\overline{u_{\ell}}(t)
\end{equation*}
thanks to the homogeneous Neumann boundary condition and the fact that $(1,\ldots,1)^T$ is a left eigenvector with eigenvalue zero of $A_i$ since $A_i$ is a reaction matrix. Thus, we have
	\begin{equation}\label{massTar}
				\sum_{k=1}^{N_i}\overline{u}_{L[i-1]+k}(t) = \sum_{k=1}^{N_i}\overline{u}_{L[i-1]+k,0} + \sum_{k=1}^{N_i}\sum_{\ell=1}^{L[i-1]}a_{L[i-1]+k,\ell}\int_{0}^{t}\overline{u_{\ell}}(s)ds.
	\end{equation}
{Given that the right hand side of \eqref{massTar} should be zero for all times $t>0$, then 
$\overline{u}_{L[i-1]+k}(t)=0$ for all $k=1,\ldots,N_i$ and for all $t>0$ and $X_{i,\infty}=0$ and the statement of the Theorem 
holds trivially.}

{Otherwise, if the right hand side of \eqref{massTar} is positive for some time $t>0$, then assumption 
\eqref{positivemass} is satisfied an  $X_{i,\infty}$ is a positive equilibrium.} 
Recalling the change of notation $v_{k} = u_{L[i-1] + k}$ in \eqref{newnotation}, we denote by
$$w_k(t) = v_k(t) - v_{k,\infty} = u_{L[i-1]+k}(t) - v_{k,\infty}$$
the distance from $u_{L[i-1]+k}$ to its corresponding equilibrium state for all $k=1,2,\ldots,N_i$. It implies that $(w_k)_{k=1,\ldots,N_i}$ solves the system \eqref{TargetSystem_Recall} subject to the initial data $w_{k,0} = u_{L[i-1]+k,0} - v_{k,\infty}$ for all $k=1,2,\ldots, N_i$. We define $W_i = (w_1, w_2, \ldots, w_{N_i})$ and consider the relative entropy-like functional
\begin{equation}\label{EntropyTarget}
\mathcal{E}(W_{i}|X_{i,\infty}) = \sum_{k=1}^{N_i}\int_{\Omega}\frac{|w_k|^2}{v_{k,\infty}}dx
= \sum_{k=1}^{N_i}\int_{\Omega}\frac{|w_k - \overline{w_k}|^2}{v_{k,\infty}}dx + \sum_{k=1}^{N_i}\frac{\overline{w_k}^2}{v_{k,\infty}}
=: \mathcal{E}_1 + \mathcal{E}_2.
\end{equation}
By using again arguments of Lemma \ref{ExplicitEnDiss}, we calculate the entropy dissipation
\begin{align}\label{EnDissTarget}
	\mathcal{D}(W_{i}|X_{i,\infty}) &= -\frac{d}{dt}\mathcal{E}(W_{i}|X_{i,\infty})\nonumber\\
&= 2\sum_{i=1}^{N_i}d_{L[i-1]+k}\int_{\Omega}\frac{|\nabla w_k|^2}{v_{k,\infty}}dx  + \sum_{k,\ell=1;k<\ell}^{N_i}(b_{k,\ell}v_{\ell,\infty} + b_{\ell,k}v_{k,\infty})\int_{\Omega}\left(\frac{w_{k}}{v_{k, \infty}} - \frac{w_{\ell}}{v_{\ell, \infty}}\right)^2dx\nonumber\\
&\quad - 2\sum_{k=1}^{N_i}\sum_{\ell=1}^{L[i-1]}a_{L[i-1]+k,\ell}\int_{\Omega}\frac{w_k}{v_{k,\infty}}\,u_{\ell}\,dx
\end{align}
For the last term of \eqref{EnDissTarget}, we estimate
\begin{align}
\biggl|2&\sum_{k=1}^{N_i}\sum_{\ell=1}^{L[i-1]}a_{L[i-1]+k,\ell}\int_{\Omega}\frac{w_k}{v_{k,\infty}}\,u_{\ell}\,dx\biggr|\nonumber \\
&\leq 2\sum_{k=1}^{N_i}\sum_{\ell=1}^{L[i-1]}a_{L[i-1]+k,\ell}\int_{\Omega}\frac{|w_k - \overline{w_k}|}{v_{k,\infty}}|u_{\ell}|dx\
	+ 2\sum_{k=1}^{N_i}\sum_{\ell=1}^{L[i-1]}a_{L[i-1]+k,\ell}\,\frac{|\overline{w_k}|}{v_{k,\infty}}|\overline{u_{\ell}}|\nonumber \\
&\leq C_{P}\sum_{k=1}^{N_i}d_{L[i-1]+k}\int_{\Omega}\frac{|w_k - \overline{w_k}|^2}{v_{k,\infty}}dx +\kappa_1\sum_{\ell=1}^{L[i-1]}\|u_{\ell}\|^2+ \kappa_2\sum_{k=1}^{N_i}\frac{\overline{w_k}^2}{v_{k,\infty}} + \kappa_3\sum_{\ell=1}^{L[i-1]}\overline{u_{\ell}}^2\nonumber \\
&\leq \sum_{k=1}^{N_i}d_{L[i-1]+k}\int_{\Omega}\frac{|\nabla w_k|^2}{v_{k,\infty}}dx + \kappa_2\sum_{k=1}^{N_i}\frac{\overline{w_k}^2}{v_{k,\infty}}+ (\kappa_1+\kappa_3)K^*e^{-\lambda^*t},\label{LastTerm}
\end{align}
{\color{black} with
	\[
		\kappa_1 = \frac{N_iL[i-1]\max\limits_{i<j}\{a_{ij}^2\}}{C_P\min\limits_{k}\{d_{L[i-1]+k}v_{k,\infty}\}}, \quad \kappa_2 = \frac 12\xi \max\limits_{k}\{v_{k,\infty}\},\quad \kappa_3 = \frac{N_iL[i-1]\max\limits_{i<j}\{a_{ij}^2\}}{\kappa_2v_{k,\infty}},
	\]
where $\kappa_2$ is chosen in such a way that the last step of the below estimate \eqref{ee1_2} is fulfilled}, and we have used $\|u_{\ell}(t)\|^2 \leq K^*e^{-\lambda^*t}$ for all $\ell = 1,\ldots, L[i-1]$ in the last estimate. By inserting \eqref{LastTerm} into \eqref{EnDissTarget}, we obtain
\begin{align}
\mathcal{D}(W_{i}|X_{i,\infty}) &\geq \sum_{i=1}^{N_i}d_{L[i-1]+k}\int_{\Omega}\frac{|\nabla w_k|^2}{v_{k,\infty}}dx+ \sum_{k,\ell=1;k<\ell}^{N_i}(b_{k,\ell}v_{\ell,\infty} + b_{\ell,k}v_{k,\infty})\left(\frac{\overline{w_{k}}}{v_{k, \infty}}- \frac{\overline{w_{\ell}}}{v_{\ell, \infty}}\right)^2\nonumber\\
  &\quad - \kappa_2\sum_{k=1}^{N_i}\frac{\overline{w_k}^2}{v_{k,\infty}}- (\kappa_1+\kappa_3) K^*e^{-\lambda^* t} =: \mathcal D_1 + \mathcal{D}_2\label{ee1}
\end{align}
where $\mathcal{D}_1$ is the term containing the gradients and $\mathcal{D}_2$ is the rest of the right hand side. It follows from Poincar\'e's inequality  that
	\begin{equation}\label{ee1_1}
			\begin{aligned}
					\mathcal D_1 \geq \sum_{i=1}^{N_i}d_{L[i-1]+k}\int_{\Omega}\frac{|\nabla w_k|^2}{v_{k,\infty}}dx\geq C_{P}\sum_{i=1}^{N_i}d_{L[i-1]+k}\int_{\Omega}\frac{|w_k-\overline{w_k}|^2}{v_{k,\infty}}dx \geq \kappa_4\mathcal{E}_1
			\end{aligned}
	\end{equation}
{\color{black}with $\kappa_4 = C_P\min\limits_{k}\{d_{L[i-1]+k}\}$}.
To control $\mathcal{E}_2$, we use arguments similar to Step 3 in the proof of Lemma \ref{EEDEstimate}. First, by using \eqref{massTar}, we have the total mass of $(w_k)_{1\leq k\leq N_i}$ is computed as,
\begin{align}
\sum_{k=1}^{N_i}\overline{w_k}(t) = \sum_{k=1}^{N_i}\overline{u}_{L[i-1]+k}(t) - \sum_{k=1}^{N_i}v_{k,\infty}\nonumber 
&= \sum_{k=1}^{N_i}\overline{u}_{L[i-1]+k,0} + \sum_{k=1}^{N_i}\sum_{\ell=1}^{L[i-1]}a_{L[i-1]+k,\ell}\int_{0}^{t}\overline{u_{\ell}}(s)ds\nonumber \\
&\quad - \sum\limits_{k=1}^{N_i}\overline{u}_{L[i-1]+k,0} - \sum\limits_{k=1}^{N_i}\sum\limits_{\ell=1}^{L[i-1]}a_{L[i-1]+k,\ell}\int_{0}^{+\infty}\overline{u_{\ell}}(s)ds\nonumber\\
&= - \sum\limits_{k=1}^{N_i}\sum\limits_{\ell=1}^{L[i-1]}a_{L[i-1]+k,\ell}\int_{t}^{+\infty}\overline{u_{\ell}}(s)ds=: -\delta(t).\label{TargetInitMass}
\end{align}
Hence, 
\begin{equation}\label{kaka}
-2\sum_{k,\ell=1;k<\ell}^{N_i}\overline{w_k}\,\overline{w_{\ell}} = -\sum_{k,\ell=1;k\neq \ell}^{N_i}\overline{w_k}\,\overline{w_{\ell} }
= \sum_{k=1}^{N_i}\overline{w_k}^2
-\sum_{k,\ell=1}^{N_i}\overline{w_k}\,\overline{w_{\ell}}
=\sum_{k=1}^{N_i}\overline{w_k}^2 - \delta^2(t).
\end{equation}
By using \eqref{e11} and \eqref{kaka}, we estimate
\begin{align}
\mathcal D_2 &\geq \xi\sum_{k,\ell=1; k<\ell}^{N_i}\left(\frac{\overline{w_k}}{v_{k,\infty}} - \frac{\overline{w_\ell}}{v_{\ell,\infty}}\right)^2 - \kappa_2\sum_{k=1}^{N_i}\frac{\overline{w_k}^2}{v_{k,\infty}} - (\kappa_1+\kappa_3)K^*e^{-\lambda^*t}\nonumber\\
&\geq -2\xi\max\limits_{k<\ell}\{v_{k,\infty}v_{\ell,\infty}\}\sum_{k,\ell=1;k<\ell}^{N_i}\overline{w_k}\,\overline{w_{\ell}} - \kappa_2\sum_{k=1}^{N_i}\frac{\overline{w_k}^2}{v_{k,\infty}} - (\kappa_1+\kappa_3)K^*e^{-\lambda^*t}\nonumber\\
&= \xi\max\limits_{k<\ell}\{v_{k,\infty}v_{\ell,\infty}\}\left(\sum_{k=1}^{N_i}\overline{w_k}^2 - \delta^2\right) - \kappa_2\sum_{k=1}^{N_i}\frac{\overline{w_k}^2}{v_{k,\infty}} - (\kappa_1+\kappa_3)K^*e^{-\lambda^*t}\nonumber \\
&\geq \frac 12\xi \max\limits_{k}\{v_{k,\infty}\}\sum_{k=1}^{N_i}\frac{\overline{w_k}^2}{v_{k,\infty}} - \xi \max\limits_{k<\ell}\{v_{k,\infty}v_{\ell,\infty}\}\delta^2 - (\kappa_1+\kappa_3)K^*e^{-\lambda^* t}\label{ee1_2} 
\end{align}
for $\varepsilon>0$ is sufficiently small. It follows from \eqref{TargetInitMass} and $\overline{u_\ell} \leq{\|u_{\ell}\|} \leq \sqrt{K^*}e^{-\lambda^*t/2}$ that
\begin{equation*}
\delta^2 \leq N_iL[i-1]\max\limits_{i<j}\{a_{ij}^2\}\sum_{\ell=1}^{L[i-1]}\left(\int_{t}^{+\infty}\overline{u_{\ell}}(s)ds\right)^2
\leq  \kappa_4e^{-\lambda^*t}
\end{equation*}
with {\color{black}$\kappa_4 = 4K^*N_iL[i-1]^2\max\limits_{i<j}\{a_{ij}^2\}(\lambda^*)^{-2}$}. Hence, \eqref{ee1_2} implies that
\begin{equation}\label{keke}
	\mathcal D_2 \geq \frac 12\xi \max\limits_{k}\{v_{k,\infty}\}\sum_{k=1}^{N_i}\frac{\overline{w_k}^2}{v_{k,\infty}} - \max\{\kappa_4\xi\max\limits_{k<\ell}\{v_{k,\infty}v_{\ell,\infty}\}, (\kappa_1+\kappa_3) K^*\}e^{-\lambda^*t} = \kappa_5\mathcal{E}_2 - \kappa_6e^{-\lambda^*t}.
\end{equation}
Combining \eqref{keke} and \eqref{ee1_1} yields
\begin{equation}\label{ee2} 
			\mathcal{D}(W_{i}| X_{i,\infty}) \geq \min\{\kappa_4, \kappa_5\}\mathcal{E}(W_{i}| X_{i,\infty}) -  \kappa_6e^{-\lambda^*t}.
\end{equation}
Therefore, by applying a classic Gronwall lemma,
\begin{equation}\label{ee4} 
		\mathcal{E}(W_{i}(t)|X_{i,\infty}) \leq K_ie^{-\lambda_i t} \qquad \text{ for all } t\geq 0
\end{equation}
{\color{black}with $\lambda_i = \min\{\kappa_4, \kappa_5,\lambda^* \}$ and $K_i = 2\max\{\mathcal{E}(X_{i,0}|X_{i,\infty}), \kappa_6\}$}. This completes the proof of the Theorem.
\end{proof}

\vskip 0.5cm
\noindent{\bf Acknowledgements.} We would like to thank the anonymous referee for his valuable comments and suggestions, which improve the presentation of the paper.

The third author is supported by International Research Training Group IGDK 1754. This work has partially been supported by NAWI Graz.


\end{document}